\documentclass{article}
\usepackage[utf8]{inputenc}
\usepackage{tikz}
\usepackage[font=footnotesize]{caption} 
\usepackage{fullpage}
\usepackage{hyperref,verbatim}
\usepackage{enumitem,bm}
\usepackage{amsthm,amsmath,amssymb}
\newtheorem{conj}{Conjecture}
\newtheorem{defi}{Definition}
\newtheorem{corollary}[defi]{Corollary}
\newtheorem{theorem}[defi]{Theorem}
\newtheorem{lem}[defi]{Lemma}
\newtheorem{cor}[defi]{Corollary}
\newtheorem{obs}[defi]{Observation}
\newtheorem{prop}{Proposition}
\newtheorem{question}{Question}
\newtheorem{Example}[defi]{Example}

\newtheorem{claim}[defi]{Claim}
\newtheorem*{cereceda}{Cereceda's Conjecture}
\newtheorem*{prop2}{Proposition~2}

\newcommand*{\myproofname}{Proof}
\newenvironment{claimproof}[1][\myproofname]{\begin{proof}[#1]}{\end{proof}}

\def\aftermath{\par\vspace{-\belowdisplayskip}\vspace{-\parskip}\vspace{-\baselineskip}}

\DeclareMathOperator{\rad}{rad}
\DeclareMathOperator{\ch}{ch}
\DeclareMathOperator{\mad}{mad}
\DeclareMathOperator{\degen}{degen}
\DeclareMathOperator{\comp}{comp}
\DeclareMathOperator{\diam}{diam}
\DeclareMathOperator{\ecc}{ecc}
\DeclareMathOperator{\dist}{dist}
\renewcommand{\a}{\alpha}
\renewcommand{\b}{\beta}
\renewcommand{\c}{\gamma}
\renewcommand{\d}{\delta}
\def\vc{\overrightarrow}
\renewcommand{\deg}{d}

\def\C{\mathcal{C}}
\newcommand{\Mod}[1]{\ (\mathrm{mod}\ #1)}
\newcommand\ceil[1]{\left\lceil#1\right\rceil}
\newcommand\floor[1]{\left\lfloor#1\right\rfloor}

\title{Optimally Reconfiguring\\ List and Correspondence Colourings}

\author{
Stijn Cambie\thanks{Department of Mathematics, Radboud University Nijmegen, Netherlands and Mathematics Institute, University of Warwick, UK. E-mail: {\tt stijn.cambie@hotmail.com}, supported by a Vidi grant (639.032.614) of the Netherlands Organisation for Scientific Research (NWO) and the UK Research and Innovation Future Leaders Fellowship MR/S016325/1. Current affiliation: Department of Computer Science, KU Leuven Campus Kulak-Kortrijk, 8500 Kortrijk, Belgium.}
\and
Wouter Cames van Batenburg\thanks{
Delft University of Technology, Delft Institute of Applied Mathematics;
\texttt{w.p.s.camesvanbatenburg@tudelft.nl}}
\and 
Daniel W. Cranston\thanks{Virginia Commonwealth University, Department of
Computer Science;
\texttt{dcranston@vcu.edu}}
}

\begin{document}

\maketitle

\begin{abstract}
The reconfiguration graph $\C_k(G)$ for the $k$-colourings of a graph $G$ has a
vertex for each proper $k$-colouring of $G$, and two vertices of $\C_k(G)$ are adjacent
precisely when those $k$-colourings differ on a single vertex of
$G$.  Much work has focused on bounding the maximum value of $\diam \C_k(G)$
over all $n$-vertex graphs $G$.  We consider the analogous problems for list
colourings and for correspondence colourings.
We conjecture that if $L$ is a list-assignment for a graph $G$ with $|L(v)|\ge
d(v)+2$ for all $v\in V(G)$, then $\diam \C_L(G)\le n(G)+\mu(G)$.
We also conjecture that if $(L,H)$ is a correspondence cover for a graph $G$ with
$|L(v)|\ge d(v)+2$ for all $v\in V(G)$, then $\diam \C_{(L,H)}(G)\le n(G)+\tau(G)$.
(Here $\mu(G)$ and $\tau(G)$ denote the matching number and vertex cover number
of $G$.)
For every graph $G$, we give constructions showing that both conjectures are
best possible, which also hints towards an exact form of Cereceda's Conjecture for regular graphs.  Our first main result proves the upper bounds (for the list and
correspondence versions, respectively) $\diam \C_L(G)\le n(G)+2\mu(G)$ and
$\diam \C_{(L,H)}(G)\le n(G)+2\tau(G)$.
Our second main result proves that both conjectured bounds hold, whenever
all $v$ satisfy $|L(v)|\ge 2d(v)+1$. 
We conclude by proving one or both conjectures for various classes of graphs such
as complete bipartite graphs, subcubic graphs, cactuses, and graphs with bounded
maximum average degree.
\end{abstract}

\section{Introduction}

In this paper we study questions of transforming one proper colouring of a graph
$G$ into another, by a sequence of \emph{recolouring steps}.  Each step recolours a
single vertex, and we require that each resulting intermediate colouring is
also proper.  Our work fits into the broader context of \emph{reconfiguration}, in
which some object (in our case a proper colouring) is transformed into another
object of the same type, via a sequence of small changes, and we require that
after each change we again have an object of the prescribed type (for us a
proper colouring).  It is natural to pose
reconfiguration questions for a wide range of objects: proper colourings,
independent sets, dominating sets, maximum matchings, spanning
trees, and solutions to 3-SAT, to name a few.  For each type of object, we must define 
allowable changes between successive objects in the sequence (in our case,
recolouring a single vertex).  Typically, we ask four types of
questions. (1) Given objects $\a$ and $\b$, is there a reconfiguration sequence
from $\a$ to $\b$? (2) If the first question is answered yes, what is the
length of a shortest such sequence? (3) Is the first question answered yes
for every pair of objects $\a$ and $\b$? 
(4) If yes, what is the maximum value of $\dist(\a,\b)$ over all $\a$ and $\b?$
For an introduction to reconfiguration,
we recommend surveys by van den Heuvel~\cite{vdH-survey} and Nishimura~\cite{nishimura}.
Most of our definitions and notation are standard, but for completeness we include many of them at the start of Section~\ref{defs:sec}.

For a graph $G$ and a positive integer $k$, the \emph{$k$-colouring reconfiguration
graph} of $G$, denoted $\C_k(G)$, has as its vertices all proper $k$-colourings
of $G$ and two vertices of $\C_k(G)$ are adjacent if their corresponding
colourings differ on exactly one vertex of $G$.  Our goal here is to study the diameter 
of this reconfiguration graph, denoted $\diam \C_k(G)$.  In general, $\C_k(G)$
may be disconnected, in which case its diameter is infinite.  For example
$\C_2(C_{2s})$ consists of two isolated vertices (here $C_{2s}$ is a cycle of
length $2s$).  More strongly, for each integer $k\ge 2$ there exist $k$-regular
graphs $G$ such that $\C_{k+1}(G)$ has isolated vertices. The simplest example
is the clique $K_{k+1}$, but this is also true for every $k$-regular graph with a
$(k+1)$-colouring $\a$ such that all $k+1$ colours appear on each closed
neighbourhood; such colourings are called \emph{frozen}.

To avoid a colouring $\a$ being frozen, some vertex $v$ must have
some colour unused by $\a$ on its closed neighbourhood, $N[v]$.  And to avoid $\C_k(G)$ being
disconnected, every induced subgraph $H$ must contain such a
vertex $v$ with some colour unused by $\a$ on $N[v]\cap V(H)$.  Thus, it
is natural to consider the degeneracy of $G$, denoted $\degen(G)$.  

Our examples
of frozen colourings above show that, if we aim to have $\C_k(G)$ connected,
then in general it is not enough to require $k\ge \degen(G)+1$.  However, an easy
inductive argument shows that a slightly stronger condition is
sufficient: If $k\ge \degen(G)+2$, then $\C_k(G)$ is connected.
(Earlier, Jerrum~\cite{jerrum} proved that $k\ge \Delta(G)+2$ suffices, but the
arguments are similar.)
This inductive proof only
yields that $\diam \C_k(G)\le 2^{|V(G)|}$.  But
Cereceda~\cite[Conjecture~5.21]{cereceda2007mixing} conjectured something
much stronger.

\begin{cereceda}
\label{cereceda-conj}
For an $n$-vertex graph $G$ with $k \ge \degen(G) + 2$, the diameter of
$\C_k(G)$ is $O(n^2)$.
\end{cereceda}

Bousquet and Heinrich~\cite{bousquet2019polynomial} proved that
$\diam \C_k(G) = O_d(n^{\degen(G)+1})$, which is
the current best known bound.
When $k \ge \Delta(G)+2$, Cereceda~\cite[Proposition~5.23]{cereceda2007mixing} proved
that $\diam \C_k(G) = O(n\Delta) = O(n^2)$. In particular,
Cereceda's conjecture
is
true for regular graphs.  But, as we show here, if $k\ge \Delta(G)+2$, then
in fact we have the stronger bound $\diam \C_k(G)\le 2n$, and we conjecture $\diam \C_k(G)\le \lfloor 3n/2 \rfloor$.

A folklorish result observes that $\diam \C_k(G)=O(n)$ when $k$ is large relative to $\Delta(G)$
(for example, see \cite[Section~3.1]{MN20}).
But until now, it seems that no one has investigated exact values of the
diameter (say, when $k\ge \Delta(G)$).  This is the main goal of our paper.  
Before stating our two main conjectures, and our results supporting them, we
present an easy lower bound on $\C_k(G)$ in terms of the matching number $\mu(G)$.

\begin{prop}
\label{lower-bound-prop}
For a graph $G$, if $k\ge 2\Delta(G)$, then $\diam\C_k(G) \ge n(G)+\mu(G)$.
\end{prop}

\begin{proof}
Let $M$ be a maximum matching in $G$.  Form $\widehat{G}$ from $G$ as follows.  If
$vw,xy\in M$ and $wx\in E(G)$, then add $vy$ to $\widehat{G}$.  Note that
$\Delta(\widehat{G})\le
2\Delta(G)-1$.  So $\chi(\widehat{G})\le 2\Delta(G)$.  Let $\a$ be a $2\Delta(G)$-colouring
of $\widehat{G}$.  Form $\b$ from $\a$ by swapping colours on endpoints of each edge
in $M$ and, for each $v\in V(G)$ not saturated by $M$, picking $\b(v)$
outside of $\{\a(v)\}\cup \bigcup_{w\in N(v)}\b(w)$.  To recolour $G$
from $\a$ to $\b$, every vertex must be recoloured.  Further, for each edge
$e\in M$, the first endpoint $v$ of $e$ to be recoloured must initially receive a
colour other than $\b(v)$, so must be recoloured at least twice.  Thus, $\diam
\C_k(G)\ge n(G)+\mu(G)$, as desired.
\end{proof}

At the end of Section~\ref{defs:sec}, we mention various special cases in which
the hypothesis of Proposition~\ref{lower-bound-prop} can be weakened. We pose it as an open question whether the bound in
Proposition~\ref{lower-bound-prop} holds for all $k\ge \Delta(G)+2$.

We study this reconfiguration problem in the more general contexts of list
colouring and correspondence colouring, both of which we define in 
Section~\ref{defs:sec}.  These more general contexts offer the added advantage
of enabling us to naturally prescribe fewer allowed colours for vertices of
lower degree.  Analogous to $\C_k(G)$, for a graph $G$ and a list-assignment $L$ or
correspondence cover $(L,H)$ for $G$, we define the \emph{$L$-reconfiguration graph
$\C_L(G)$} or \emph{$(L,H)$-reconfiguration graph $\C_{(L,H)}(G)$} of $G$. 
Generalizing our constructions of frozen $k$-colourings
above, it is easy to show that $\C_L(G)$ and $\C_{(L,H)}(G)$ can contain frozen
$L$-colourings (and thus be disconnected) if we require only that all $v\in V(G)$
satisfy $\lvert L(v) \rvert =\deg(v)+1$.  Thus, we adopt a
slightly stronger hypothesis: all $v \in V(G)$ satisfy $\lvert L(v) \rvert \ge
\deg(v)+2$.  Now we can state our two main conjectures.

\begin{conj}[List Colouring Reconfiguration Conjecture]
\label{conj:main_list}
For a graph $G$, if $L$ is a list-assignment such that $\lvert L(v) \rvert
\ge \deg(v)+2$ for every $v \in V(G)$, then $\diam \C_L(G) \le n(G)+\mu(G)$.
\end{conj}

\begin{conj}[Correspondence Colouring Reconfiguration Conjecture]
\label{conj:main_DP}
For a graph $G$, if $(L,H)$ is a correspondence cover such that $\lvert L(v) \rvert
\ge \deg(v)+2$ for every $v \in V(G)$, then $\diam \C_{(L,H)}(G) \le n(G)+\tau(G)$.
\end{conj}

Here $\tau(G)$ denotes the vertex cover number of $G$. For brevity, we often call these conjectures the 
List Conjecture and the
Correspondence Conjecture.
Our aim in this paper is to provide significant evidence for both conjectures.
Due to Proposition~\ref{lower-bound-prop}, the List Conjecture
is best possible, when the lists are large enough.  
We will soon give easy constructions showing that both conjectures are best
possible whenever all $v$ satisfy $|L(v)|\ge d(v)+2$.  But we defer these
constructions until Section~\ref{defs:sec}, where we formally define list and
correspondence colourings.

\begin{prop2}
For every graph $G$ (i) there exists a list-assignment $L$ such that
$|L(v)|=d(v)+2$ for all $v$ for which $\diam \C_L(G)=n(G)+\mu(G)$ and (ii) there
exists a correspondence cover $(L,H)$ such that $|L(v)|=d(v)+2$ for all
$v$ for which $\diam \C_{(L,H)}(G)=n(G)+\tau(G)$.
\end{prop2}

For both the 
\hyperref[conj:main_list]{List Conjecture} 
and the 
\hyperref[conj:main_DP]{Correspondence Conjecture}, it is trivial to
construct examples that need at least $n(G)$ recolourings.  We simply require
that $\a(v)\ne \b(v)$ for all $n(G)$ vertices $v$.  So we view these conjectured
upper bounds as consisting of a ``trivial'' portion, $n(G)$ recolourings, and a
``non-trivial'' portion, $\mu(G)$ or $\tau(G)$ recolourings.  
We give two partial results toward each conjecture.
Our first result proves both conjectures up to a factor of 2 on the non-trivial
portions of these upper bounds.

\begin{theorem}
\label{factor2:thm}
(i) For every graph $G$ and list-assignment $L$ with $|L(v)|\ge d(v)+2$ for all
$v\in V(G)$ we have $\diam \C_L(G)\le n(G)+2\mu(G)$.  (ii) For every graph $G$
and correspondence cover $(L,H)$ with $|L(v)|\ge d(v)+2$ for all $v\in V(G)$ we
have $\diam \C_{(L,H)}(G)\le n(G)+2\tau(G)$.
\end{theorem}

Bousquet, Feuilloley, Heinrich, and Rabie~\cite[following Question 1.3]{BFHR22}
asked about the diameter of $\C_k(G)$ when $k=\Delta(G)+2$.
In this case, a result of Bonamy and Bousquet~\cite[Theorem~1]{BB18} implies,
for an $n$-vertex graph $G$ and $k=\Delta(G)+2$, that
$\diam~\C_k(G)=O(\Delta(G)n)$.  The authors of~\cite{BFHR22}
asked whether it is possible to remove this dependency on $\Delta(G)$.
We answer their question affirmatively.
Always $\mu(G)\le n(G)/2$, so Theorem~\ref{factor2:thm}(i) implies
that $\diam \C_k(G)\le 2n(G)$ when $k=\Delta(G)+2$.
%

To complement Theorem~\ref{factor2:thm}, we prove both conjectured
upper bounds when $|L(v)|$ is sufficiently large.

\begin{theorem}
\label{big-lists:thm}
(i) For every graph $G$ and list-assignment $L$ with $|L(v)|\ge 2d(v)+1$ for all
$v\in V(G)$ we have $\diam \C_L(G)\le n(G)+\mu(G)$.  (ii) For every graph $G$
and correspondence cover $(L,H)$ with $|L(v)|\ge 2d(v)+1$ for all $v\in V(G)$ we
have $\diam \C_{(L,H)}(G)\le n(G)+\tau(G)$.
\end{theorem}

In Section~\ref{defs:sec} we present definitions and notation, as well as some
easy lower bounds on diameters of reconfiguration graphs.
In Section~\ref{main-thms-list:sec} we prove the list colouring portions of
Theorems~\ref{factor2:thm} and~\ref{big-lists:thm}; and in
Section~\ref{main-thms-corr:sec} we prove the correspondence colouring portions.
In Section~\ref{trees:sec} we prove more precise results when $G$ is a tree (in
this case the correspondence problem is identical to the list problem). 

In Sections~\ref{classes-list:sec} and~\ref{classes-correspondence:sec}, we
conclude by proving one or both conjectures exactly for various graph classes,
such as complete bipartite graphs, subcubic graphs, cactuses, and graphs
with low maximum average degree.
These graphs are all sparse or have low chromatic number. 
On the other end of the sparsity spectrum,  
the \hyperref[conj:main_list]{List Conjecture} is
also true for all complete graphs, due to an argument of Bonamy and
Bousquet~\cite[Lemma~5]{BB18}.

\section{Definitions, Notation, and Easy Lower Bounds}
\label{defs:sec}

For a graph $G=(V,E)$, we denote its order 
by $n(G)$.
A \emph{matching} is a set of vertex disjoint edges in $G$.  
A matching $M$ of $G$ is \emph{perfect} if it saturates every vertex of $G$, and
$M$ is \emph{near-perfect} if it saturates every vertex of $G$ but one.
A
\emph{vertex cover} $S$ is a vertex subset such that every edge of $G$ has at least
one endpoint in $S$.  The \emph{matching number} $\mu(G)$ of $G$ is the size of
a largest matching.  The \emph{vertex cover number} $\tau(G)$ of $G$ is the
size of a smallest vertex cover.  (Recall that if $G$ is bipartite, then
$\mu(G)=\tau(G)$.)

The \emph{distance} $\dist(v,w)$ between two vertices $v,w \in V(G)$ is the length
of a shortest path in $G$ between $v$ and $w$.
The \emph{eccentricity} of a vertex $v$, denoted $\ecc(v)$,
is $\max_{w \in V} \dist(v,w).$ 
The \emph{diameter} $\diam(G)$ of $G$ is $\max_{v,w \in V} \dist(v,w)$, which is
equal to $\max_{v \in V} \ecc(v)$, while the \emph{radius} $\rad(G)$ of
$G$ is $\min_{v \in V} \ecc(v)$.
A diameter can also refer to a shortest path between $v$ and $w$ for which $\dist(v,w)=\diam(G).$
A graph $G$ is \emph{$k$-degenerate} if every non-empty subgraph $H$ contains a
vertex $v$ such that $\deg_H(v)\le k$.  The \emph{degeneracy} of $G$ is the
minimum $k$ such that $G$ is $k$-degenerate.

A directed graph or \emph{digraph} $D=(V,A)$ is analogous to a graph, except
that each edge is directed.
The directed edges in $A$ are  ordered pairs of vertices, and are called \emph{arcs}.

\begin{defi}
For a digraph $D$, the vertex cover number $\tau(D)$
and matching number $\mu(D)$ are defined to be $\tau(H_D)$ and $\mu(H_D)$, where $H_D$ is the undirected graph with vertex set $V(D)$ whose edges
correspond to the bidirected edges in $D$.
\end{defi}

\emph{Colourings} are mappings from $V$ to ${\mathbb N}$, and we denote them by
greek letters such as $\a, \b$, and $\gamma$.
A colouring $\a$ of $G$ is \emph{proper} if for every edge $vw\in E(G)$ we have $\a(v)\ne \a(w)$. 
Let $[k]:=\{1,2, \ldots, k\}$. A \emph{$k$-colouring} is a colouring using at most $k$ colours, which are generally from the set $[k]$.
The \emph{chromatic number} $\chi(G)$ of $G$ is the smallest $k$ such that $G$ has a proper $k$-colouring.

A \emph{list-assignment} $L$, for a graph $G$, assigns to each $v\in V(G)$ a set
$L(v)$ of natural numbers (``list'' of allowable colours).
A \emph{proper $L$-colouring} is a proper colouring $\a \colon V(G)\to {\mathbb N}$ such that $\a(v)\in L(v)$ for every $v\in V(G)$.
The \emph{list chromatic number} (or \emph{choice number}) $\chi_\ell(G)$ is
the least $k$ such that $G$ admits a proper $L$-colouring 
whenever  every $v \in V(G)$ satisfies $\lvert L(v) \rvert \ge k$.

A \emph{correspondence cover} $(L,H)$ for a graph $G$ assigns to each vertex
$v\in V(G)$ a set $L(v)$ of colours $\{(v,1),\ldots,$ $(v,f(v))\}$ and to each
edge $vw\in E(G)$
a matching between $L(v)$ and $L(w)$.  (In this paper, we typically consider
either $f(v)=d(v)+2$ or $f(v)=2d(v)+1$ for all $v\in V(G)$.)  Given a
correspondence cover $(L,H)$, an \emph{$(L,H)$-colouring} of $G$ is a function
$\a$ such that $\a\in L(v)$ for all $v$ and whenever $vw\in E(G)$ the edge
$\a(v)\a(w)$ is not an edge of the matching assigned to $vw$.  Here $H$ denotes
the union of the matchings assigned to all edges of $G$. (It is easy to
check that correspondence colouring generalises list colouring.)

The \emph{reconfiguration graph} $\C_k(G)$ has as its vertices the proper
$k$-colourings of $G$, and two $k$-colourings $\a$ and $\b$ are adjacent in
$\C_k(G)$ if they differ on exactly one vertex of $G$. 
In particular, if $k<\chi(G)$, then $\C_k(G)$ has no vertices.
(The reconfiguration graph was first defined in~\cite{CvdHJ08}, where it was
called the $k$-colour graph.) The distance between $k$-colourings $\a$
and $\b$ is at most $j$ if we can form $\b$, starting from $\a$, by recolouring
at most $j$ vertices, one at a time, so that after each recolouring the current
$k$-colouring of $G$ is proper.
Similarly, for a graph $G$ and list-assignment $L$, the \emph{reconfiguration
graph} $\C_L(G)$, or reconfiguration graph of the $L$-colourings of $G$, is the
graph whose vertices are the proper $L$-colourings of $G$.  Again, two
$L$-colourings $\a$ and $\b$ are adjacent in $\C_L(G)$ if they differ on
exactly one vertex.  This extension was first defined in~\cite{BC}.
For a correspondence cover $(L,H)$, the \emph{reconfiguration graph} $\C_{(L,H)}(G)$ is
defined analogously; its vertices are the correspondence colourings of $G$ and
two vertices of $\C_{(L,H)}(G)$ are adjacent precisely when their colourings
differ on exactly one vertex of $G$.
%
We will also use the associated colour-shift digraph for a graph $G$ and two
proper colourings, which is defined as follows.

\begin{defi}
Let $G=(V,E)$ be a graph and $\a$ and $\b$ be two proper colourings of $G$. We
define an associated colour-shift digraph $D_{\a, \b}=(V,A)$ where $\vc{vw}
\in A$ if and only if $vw \in E$ and $\b(v)=\a(w)$.
\end{defi}

\subsection{Improved Lower Bounds}

In this subsection we prove lower bounds which show that our upper bounds in the
rest of the paper are sharp or nearly sharp.  These will not be explicitly
needed elsewhere, so the impatient reader should feel free to skip to
Section~\ref{main-thms-list:sec}, where we prove Theorems~\ref{factor2:thm}(i)
and \ref{big-lists:thm}(i).

\begin{obs}
\label{lower-bound:obs}
If $G$ is a graph with list colourings $\a$ and $\b$, then $\dist(\a,\b)
\ge \mu(D_{\a, \b})+\sum_{v \in V} \mathbf{1}_{\a(v) \not= \b(v)}$.
\end{obs}
\begin{proof}
Whenever $\a(v) \not= \b(v)$, vertex $v$ must be recoloured.
Further, if $v$ and $w$ are neighbours for which $\a(v)=\b(w)$ and
$\a(w)=\b(v)$, then we cannot recolour both $v$ and $w$ only once, since after
every recolouring step the resulting colouring must be proper.
So at least $\sum_{v \in V} \mathbf{1}_{\a(v) \not= \b(v)}$ vertices must be
recoloured, and at least $\mu(D_{\a, \b})$ vertices must be recoloured at least
twice. 
\end{proof}

Does the lower bound in Observation~\ref{lower-bound:obs} always hold with
equality?  Our next example shows that it does not; see
Figure~\ref{fig:C4(C4)reconfiguration}.
Specifically, this example exhibits, for the 4-cycle, colourings $\a$ and
$\b$ such that $\mu(D_{\a,\b})<\mu(G)$ but still $\dist(\a,\b)=n(G)+\mu(G)$.

\begin{Example}
Let $G=C_4$ and denote $V(G)$ by $[4]$. If $\a(i)=i$ and $\b(i)\equiv
i+1 \pmod 4$ and the colour set is precisely $[4]$, then transforming $\a$ to $\b$ uses at least 6 recolourings,
i.e. $\dist(\a, \b)= 6$.
	Part of this reconfiguration graph is shown in 
Figure~\ref{fig:C4(C4)reconfiguration}.
\end{Example}
\begin{proof}
Starting from $\a$, each vertex $i$ has a single colour with which it can be
recoloured.  But each possible recolouring creates a colouring $\a'$ for which
Observation~\ref{lower-bound:obs} gives $\dist(\a',\b)\ge 5$.
\end{proof}

\begin{figure}[!h]
\centering
\begin{tikzpicture}[scale=.50]
\tikzstyle{mythick}=[line width=0.8mm] 
\tikzstyle{mythicker}=[line width=2.5mm] 
\tikzset{every node/.style=uStyle}
\def\Srad{.25} 
\def\Brad{1.9} 
\def\len{0.8cm} 
\def\myshift{2in} 
\def\myblue{blue!60!white} 

\draw[mythicker] (0,\myshift) -- (3*\myshift,\myshift) -- (3*\myshift,0) --
(\myshift,0) (0,0) -- (\myshift,\myshift);

\newcommand\ColorNode[6]
{
\draw[mythick, fill=white] (#1,#2) circle (\Brad);
\draw[mythick] (#1,#2) ++ (\len,\len) --++ (-2*\len,0) --++ (0,-2*\len) --++ (2*\len,0) -- cycle;
\foreach \x/\y/\col in {1/1/#3, -1/1/#4, -1/-1/#5, 1/-1/#6}
\draw[fill=\col, thick] (#1,#2) ++ (\x*\len,\y*\len) circle (\Srad);
}

\ColorNode{0}{0}{red}{\myblue}{red}{\myblue}
\ColorNode{\myshift}{0}{\myblue}{yellow}{green}{red}
\ColorNode{2*\myshift}{0}{\myblue}{yellow}{green}{yellow}
\ColorNode{3*\myshift}{0}{\myblue}{yellow}{red}{yellow}
\ColorNode{0}{\myshift}{red}{\myblue}{yellow}{green}
\ColorNode{\myshift}{\myshift}{red}{\myblue}{red}{green}
\ColorNode{2*\myshift}{\myshift}{red}{yellow}{red}{green}
\ColorNode{3*\myshift}{\myshift}{\myblue}{yellow}{red}{green}
\end{tikzpicture}

\caption{Part of the reconfiguration graph $\C_4(C_4)$}
\label{fig:C4(C4)reconfiguration}
\end{figure}
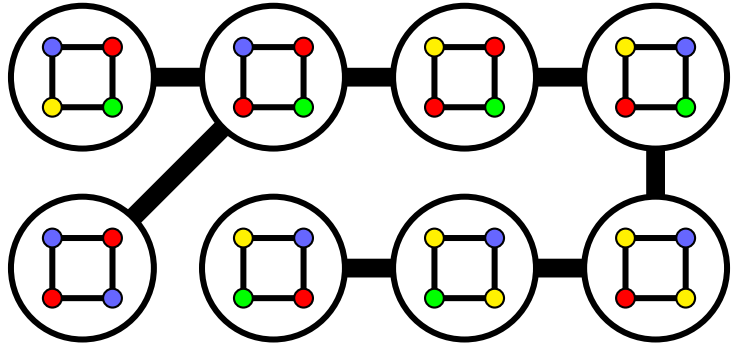

Next we present the (previously promised) proof of
Proposition~\ref{lower-bounds:prop}.  For
easy reference, we restate it.

\begin{prop}
\label{lower-bounds:prop}
For every graph $G$
(i) there exists a list-assignment $L$ such that
$|L(v)|=\deg(v)+2$ for all $v$ for which $\diam \C_L(G)=n(G)+\mu(G)$ and (ii) there
exists a correspondence cover $(L,H)$ such that $|L(v)|=\deg(v)+2$ for all $v$
for which $\diam \C_{(L,H)}(G)=n(G)+\tau(G)$.
\end{prop}
\begin{proof}
For a graph $G$, fix a maximum matching $M$.  Assign to each vertex $v$ a list of
$\deg(v)+2$ colours such that $|L(v)\cap L(w)|=2$ if $vw\in M$ and otherwise
$|L(v)\cap L(w)|=0$.  Pick $\a$ and $\b$ such that for each $vw\in M$, we have
$\a(v)=\b(w)$ and $\a(w)=\b(v)$, and for each $v\in V(G)$ we have $\a(v) \not=\b(v).$
The lower bound holds by Observation~\ref{lower-bound:obs}.
The upper bound is trivial, since $L(x)\cap L(y)=\emptyset$ for all $xy\in
E(G)\setminus M$.  This proves (i).


Let $(L,H)$ be a correspondence cover of $G$ such that
$L(v)=\{(v,1)\ldots(v,\deg(v)+2)\}$ for every $v \in V(G)$. 
For every $vw\in E(G)$, let the matching (from $H$) consist of the two edges $(v,1)(w,2)$ and $(v,2)(w,1)$; otherwise, $(v,i)$ and $(w,j)$ are unmatched.
Let $\a(v):=(v,1)$ and $\b(v):=(v,2)$ for all $v\in V(G)$.
Starting from $\a$, we can recolour a vertex $v$ with $\b(v)$ only if all
neighbours of $v$ have already been recoloured.  Thus, the set of vertices recoloured
only once must be an independent set; equivalently, the set of vertices
recoloured at least twice must be a vertex cover. 
So we must use at least $n(G)+\tau(G)$ recolourings, which proves the lower
bound.  For the upper bound, fix a minimum vertex cover $S$.  First recolour
each $v\in S$ with $(v,3)$; afterward, recolour each vertex with
$(v,2)$, first the vertices in $V(G)\backslash S$ and then those in $S$.
So $n(G)+\tau(G)$ recolourings suffice.
This proves (ii).
\end{proof}

It is helpful to note that the proof of Proposition~\ref{lower-bounds:prop}
actually works (for both list colouring and correspondence colouring) whenever
all $v$ satisfy $|L(v)|\ge 3$. 

Clearly, the graph $\widehat{G}$ constructed in
Proposition~\ref{lower-bound-prop} depends on our choice of a perfect matching
$M$.  It is interesting to note that some choices of $M$ work quite well, while
others work rather poorly. 

\begin{Example}
\label{choose-M-example}
Figure~\ref{choose-M-fig} shows
a graph $G$ built from two cliques $K_p$ and a complete bipartite
graph $K_{p,p}$, with parts $U$ and $W$, by adding a perfect matching $M$
between one copy of $K_p$ and the vertices of $U$ and a perfect matching
between $W$ and the other copy of $K_p$.
Now $\chi(\widehat{G})=2p$ if we choose our perfect matching to be $M$. 
If $p$ is even, then we can instead choose a perfect matching $M'$ that is the
disjoint union of perfect matchings in the two copies of $K_p$ and in the
$K_{p,p}$ so that the resulting $\widehat{G}$ instead satisfies
$\chi(\widehat{G})=p+2$.
\end{Example}

\begin{figure}[!h]
\centering
\begin{tikzpicture}[thick, scale=.65]
\tikzstyle{uStyle}=[shape = circle, minimum size = 4.5pt, inner sep = 0pt,
outer sep = 0pt, draw, fill=white, semithick]
\tikzstyle{sStyle}=[shape = rectangle, minimum size = 4.5pt, inner sep = 0pt,
outer sep = 0pt, draw, fill=white, semithick]
\tikzstyle{lStyle}=[shape = circle, minimum size = 4.5pt, inner sep = 0pt,
outer sep = 0pt, draw=none, fill=none]
\tikzset{every node/.style=uStyle}
\def\whitemargin{.35}


\begin{scope} 
\draw
[fill=black!10!white] (1,6.75) --++ (-1.0,0) --++ (0,-6.5) --++ (2,0) --++
(0,6.5) --++ (-1.0,0);

\draw [fill=white, rounded corners=.6*\whitemargin cm] (0,6.75) --++
(-\whitemargin,0) --++ (0,-6.5) --++ (2*\whitemargin,0) --++ (0,6.5) --++ (-\whitemargin,0);

\draw [fill=white, rounded corners=.6*\whitemargin cm] (2,6.75) --++
(-\whitemargin,0) --++ (0,-6.5) --++ (2*\whitemargin,0) --++ (0,6.5) --++ (-\whitemargin,0);

\foreach \y in {1,...,6} 
\draw (0,\y) node (u\y) {};

\foreach \y in {1,...,6} 
\draw (2,\y) node (w\y) {};
\end{scope}

\begin{scope}[xscale=.3, xshift=-7cm] 
\draw[fill=black!10!white] (0,3.5) circle (3.2cm);
\foreach \ang/\name in {-63/1, -32/2, -10/3, 10/4, 32/5, 63/6}
\draw (-0.5,3.5) ++ (\ang:2.8cm) node (x\name) {};
\end{scope}

\begin{scope}[xscale=.3, xshift=13.5cm] 
\draw[fill=black!10!white] (0.5,3.5) circle (3.2cm);
\foreach \ang/\name in {-63/6, -32/5, -10/4, 10/3, 32/2, 63/1}
\draw (1,3.5) ++ (180+\ang:2.8cm) node (y\name) {};
\end{scope}

\begin{scope}[line width=1mm] 
\draw (x1) -- (x2) (x3) -- (x4) (x5)--(x6);
\draw (y1) -- (y2) (y3) -- (y4) (y5)--(y6);
\draw (u1) -- (w1) (u2) -- (w2) (u3) -- (w3) (u4) -- (w4) (u5) --
(w5) (u6) -- (w6);
\end{scope}

\foreach \i in {1,...,6} 
\draw (u\i) -- (x\i) (w\i) -- (y\i);


\begin{scope}[xshift=-11cm]

\begin{scope} 
\draw
[fill=black!10!white] (1,6.75) --++ (-1.0,0) --++ (0,-6.5) --++ (2,0) --++
(0,6.5) --++ (-1.0,0);

\draw [fill=white, rounded corners=.6*\whitemargin cm] (0,6.75) --++
(-\whitemargin,0) --++ (0,-6.5) --++ (2*\whitemargin,0) --++ (0,6.5) --++ (-\whitemargin,0);

\draw [fill=white, rounded corners=.6*\whitemargin cm] (2,6.75) --++
(-\whitemargin,0) --++ (0,-6.5) --++ (2*\whitemargin,0) --++ (0,6.5) --++ (-\whitemargin,0);

\foreach \y in {1,...,6} 
\draw (0,\y) node (u\y) {};

\foreach \y in {1,...,6} 
\draw (2,\y) node (w\y) {};
\end{scope}

\begin{scope}[xscale=.3, xshift=-7cm] 
\draw[fill=black!10!white] (0,3.5) circle (3.2cm);
\foreach \ang/\name in {-63/1, -32/2, -10/3, 10/4, 32/5, 63/6}
\draw (-0.5,3.5) ++ (\ang:2.8cm) node (x\name) {};
\end{scope}

\begin{scope}[xscale=.3, xshift=13.5cm] 
\draw[fill=black!10!white] (0.5,3.5) circle (3.2cm);
\foreach \ang/\name in {-63/6, -32/5, -10/4, 10/3, 32/2, 63/1}
\draw (1,3.5) ++ (180+\ang:2.8cm) node (y\name) {};
\end{scope}

\foreach \i in {1,...,6} 
\draw[line width=1mm] (u\i) -- (x\i) (w\i) -- (y\i);

\end{scope}

\end{tikzpicture}
\caption{A graph $G$ (shown on both left and right) consisting of a complete
bipartite graph $K_{p,p}$ in the
center, two copies of $K_p$ on the sides, and a matching from each copy of $K_p$
to one part of $K_{p,p}$. The left and right figures
specify different perfect matchings $M$ (in bold). 
On the left, $\chi(\widehat{G}) \ge \omega(\widehat{G}) = |V(K_{p,p})| = 2p =
2\Delta(G)-2$.  On the right, $\omega(\widehat{G})=\omega(G) = p$, and in fact
$\chi(\widehat{G})\le p+2=\Delta(G)+1$. (We use colours $1,\ldots,p$ on each
clique and colours $p+1$ and $p+2$ in the center.)%
\label{choose-M-fig}
}
\end{figure}
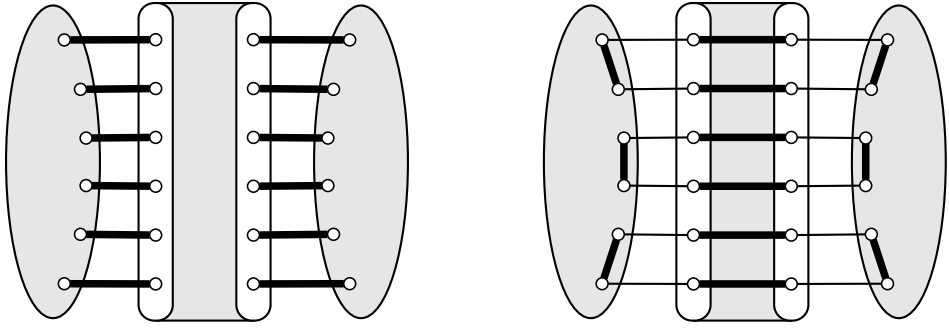

Although Example~\ref{choose-M-example} implies that our choice of $M$ can
effect $\chi(\widehat{G})$ dramatically, for many graphs, any choice of $M$ is
fine.  To conclude this section, we present various cases in which (for all
choices of $M$) we can weaken the hypothesis of Proposition~\ref{lower-bound-prop}.

\begin{prop}
For a graph $G$ and positive integer $k$ we have $$\diam~\C_k(G)\ge n(G)+\mu(G)$$
whenever {(a) $\chi(\widehat{G})\le k-1$ or (b) $\chi(\widehat{G})\le k$ and
$k\ge \Delta(G)+2$}, where $\widehat{G}$ is as in
Proposition~\ref{lower-bound-prop}.  In particular, this is true 
in the following cases.
\begin{enumerate}
\item\label{itm:1} 
{$1+\chi(G)(\chi(G)-1)\le k$ (in particular, if $G$ is bipartite and $k\ge
3$).}
\item\label{itm:2} $\Delta(G)=3$ and $k\ge \Delta(G)+2=5$.
\item 
$\chi(G)\le 3$ and $k\ge \Delta(G)+2$ (in particular,
if $G$ is outerplanar and $k\ge \Delta(G)+2$). 
\item $k\ge \Delta(G)+3$ and $G$ is triangle-free, if Reed's
Conjecture\footnote{Recall that Reed conjectured that $\chi(G)\le
\ceil{(\Delta(G)+1+\omega(G))/2}$ for every graph $G$.} holds for every graph
$\widehat{G}$ with $\omega(\widehat{G})\le 5$. 
\item $k\ge \Delta(G)+2$ and $\Delta(G)>f(\omega(G))$ for some function $f$.
\end{enumerate}
\end{prop}
\begin{proof}
Recall the definition of $\widehat{G}$
from Proposition~\ref{lower-bound-prop}.
Let $M$ be a maximum matching in $G$.  Form $\widehat{G}$ from $G$ as follows.  If
$vw,xy\in M$ and $wx\in E(G)$, then add $vy$ to $\widehat{G}$.  
To prove each case of the proposition, we construct the graph $\widehat{G}$,
{let $\a$ be a $\chi(\widehat{G})$-colouring of $\widehat{G}$ (which
is also a $\chi(\widehat{G})$-colouring of $G$)}, and form a $k$-colouring $\b$ from $\a$ by swapping the colours
on the endpoints of each edge $e$ in the specified maximum matching $M$; for each $v$ not saturated by $M$, we choose $\b(v)$ to avoid
$\a(v)\cup\bigcup_{x\in N(v)}\b(x)$.
{The latter is possible because we assumed that $k$ is bounded from below by $ \chi(\widehat{G})+1$ or by $\Delta(G)+2$.}
By Observation~\ref{lower-bound:obs}, we have $\dist(\a,\b) \ge \mu(D_{\a, \b})+\sum_{v \in V}
\mathbf{1}_{\a(v) \not= \b(v)}= \mu(G)+n(G)$.
Note that $\Delta(\widehat{G})\le
2\Delta(G)-1$.  So $\chi(\widehat{G})\le 1+\Delta(\widehat{G})\le 2\Delta(G)$.  Now we show that in each case
listed above we have $\chi(\widehat{G})\le k$ 
{ or (in the first case) $\chi(\widehat{G})\le k-1$.}
\begin{enumerate}
    \item Fix a $\chi(G)$-colouring $\a$ of $G$.  We assume that $M$ is a perfect
matching; if not, then we add a pendent edge at each vertex that is unsaturated
by $M$ and add all these new edges to $M$.    To colour $\widehat{G}$, we give
each vertex $v$ the colour $(i,j)$, where $\a(v)=i$, $\a(w)=j$, and $vw\in M$. 
It is easy to check that the resulting colouring of $\widehat{G}$ is proper.
    \item Suppose that $\Delta(G)=3$.  As noted above, we have $\Delta(\widehat{G})\le
2\Delta(G)-1=5$.  If $\Delta(\widehat{G})\le 4$, then $\chi(\widehat{G})\le 5$ as desired.
So assume instead that $\Delta(\widehat{G})=5$.  By Brooks' Theorem it suffices to show
that no component of $\widehat{G}$ is $K_6$.  Suppose, to the contrary, that there exists
$S\subseteq V(G)$ such that $\widehat{G}[S]=K_6$.  It is easy to check that $G[S]$ must
be 3-regular, and every vertex of $S$ must be saturated by $M$.  Observe that if
two edges of $M$, say $v_1v_2,v_3v_4$ lie on a 4-cycle in $G$, then
$d_{\widehat{G}}(v_i)\le 4$, for each $i\in\{1,2,3,4\}$.  Now we need only to check that
such a configuration occurs for every 3-regular graph $G$ on 6 vertices
and every perfect matching $M$.  This is straightforward to verify: $G-M$ is
either a 6-cycle or two 3-cycles; in the first case, $M$ can be added in two
(non-isomorphic) ways and in the second, $M$ can be added in only one way.
\item If $\Delta(G)\ge 4$, then we are
done by condition~\ref{itm:1}.  If $\Delta(G)=3$, then we are done by condition~\ref{itm:2}.  If $\Delta(G)=2$, then
we are done by Proposition~\ref{lower-bound-prop}. And if $\Delta(G)\le 1$, then we
are done trivially.

\item We show that if $G$ is triangle-free, then $\omega(\widehat{G})\le 5$.
Suppose, to the contrary, that $\omega(\widehat{G})\ge 6$.  Consider
$S\subseteq V(G)$ such that $\widehat{G}[S]=K_6$, and colour each edge of
$\widehat{G}[S]$ with 1 if it appears in $G$ and with 2 otherwise.  Since $G$
is triangle-free, $\widehat{G}[S]$ has no triangle coloured 1.  
Let $R(t,t)$ denote the diagonal Ramsey number for cliques of order $t$. 
Since $R(3,3)=6$, we must have a triangle in
$\widehat{G}[S]$ that is coloured 2; denote its vertices by $v,w,x$.  
Let $vv',ww',xx'$ denote the
edges of $M$ incident with $v,w,x$.  Since $vw,vx,wx\in E(\widehat{G})\setminus E(G)$,
we conclude that $v'w', v'x',w'x'\in E(G)$.  This contradicts that $G$ is
triangle-free.  Hence, $\omega(\widehat{G})\le 5$, as claimed.  Finally, by Reed's
Conjecture, $\chi(\widehat{G})\le \ceil{(\Delta(\widehat{G}) +
\omega(\widehat{G})+1)/2}\le \ceil{(2\Delta(G)-1+5+1)/2}\le \Delta(G)+3\le k$.

\item Let $t$ be a positive integer and consider any graph $G$ with
$\omega(G)=t$. Note that $\widehat{G}$ can be decomposed into two
(edge-disjoint) copies of $G$; here $\widehat{G}$ is formed from the two copies
of $G$ by identifying, for each edge $vw\in M$, the instance of $v$ in one copy
of $G$ with the instance of $w$ in the other copy.
This implies that $\omega(\widehat{G}) <R(t+1,t+1)$. 
We also know that $\Delta(\widehat{G})\le 2 \Delta(G).$
Now by e.g.~\cite[Theorem~2]{Molloy} we have
$$\chi(\widehat{G}) \le 200 R(t+1,t+1) \frac{ 2 \Delta(G) \ln\ln(2
\Delta(G))}{\ln(2
\Delta(G))}.$$
This is smaller than $\Delta(G)+2$ whenever $\Delta(G)$ is sufficiently large.
\end{enumerate}
\aftermath
\end{proof}

We also consider the graph $\tilde{G}$ formed from $G$ by contracting each edge
of a maximum matching $M$. 

\begin{prop}
For a graph $G$ and positive integer $k$ we have $$\diam~\C_k(G)\ge n(G)+\mu(G)$$
whenever $k \ge 2\chi(\tilde{G})$ for
some maximum matching of $G$.
In particular, this is
\begin{enumerate}
\item true a.a.s. for $G_{n,p}$, for fixed $p\in (0,1)$, and $k\ge \Delta(G_{n,p})+2.$
\item 
true if $G$ is $K_{t+1}$-minor free and $k\geq 2t$, provided Hadwiger's
Conjecture\footnote{Recall that Hadwiger conjectured that $\chi(G)\le t$ for
every graph $G$ with no $K_{t+1}$-minor.} holds for $K_{t+1}$-minor free
graphs. The latter is true for $t\le 5$.
\end{enumerate}
\end{prop}
\begin{proof}
Let $\tilde{\a}$ be a $\chi(\tilde{G})$-colouring of $\tilde{G}$.
{We construct a $2\chi(\tilde{G})$-colouring $\a$ of $G$ such that if $v\in V(\tilde{G})$ and
$v$ arises from contracting edge $wx\in M$ and $\tilde{\a}(v)=i$, then
$\{\a(w),\a(x)\}=\{2i-1,2i\}$; if $v$ does not arise from contracting an
edge, then we choose $\a(v)=2\tilde{\a}(v)=2i$.
Form $\b$ from $\a$ by swapping the colours on the endpoints of each edge of
$M$ and letting $\b(v):=\a(v)-1=2i-1$ for each $v \in V(G) \backslash M$.}

Hence $\diam \C_k(G) \ge \dist(\a, \b) \ge n(G)+\mu(G)$ by~Observation~\ref{lower-bound:obs}.
We now prove that in the two mentioned cases, we (a.a.s.) have $k\geq
2\chi(\tilde{G})$, regardless of the maximum matching of $G$ we choose to form $\tilde{G}$.

\begin{enumerate}
    \item Let $G:=G_{n,p}$.  
By~\cite{BCE80}, the Hadwiger number $h$ (or contraction clique number) of $G$
for fixed $p\in (0,1)$ is $\Theta\left( \frac{n}{\sqrt{\ln n}}\right)$ 
a.a.s., so $G$ is $K_{h+1}$-minor free. 
    Note that $\tilde{G}$ is also $K_{h+1}$-minor free.
    By recent progress towards Hadwiger's conjecture, culminating in~\cite{DelcourtPostle21}, this implies that (still
a.a.s.) $\chi(\tilde{G}) = O(h \ln \ln h)=O\left( \frac{n \ln \ln n}{\sqrt{\ln n}}\right)$.
    So $\chi(\tilde{G})$ is a.a.s. smaller than $\Delta(G_{n,p})+2$, which is
$np(1+o(1))=\Theta(n).$
    
    \item Note that $\tilde{G}$ is also $K_{t+1}$-minor free.  So, by
assumption, $\chi(\tilde{G})\le t$. 
\end{enumerate}
\aftermath
\end{proof}

\begin{question}\label{q:n+mu_lowerbound}
Is it true, for every graph $G$ and every $k\ge \Delta(G)+2$, that
$\diam~\C_k(G)\ge n(G)+\mu(G)$?
\end{question}

If the answer to this question is yes, then the 
\hyperref[conj:main_list]{List Conjecture} 
(if true) would imply that $\diam~\C_k(G)=n(G)+\mu(G)$ for every graph $G$ and
every $k\ge \Delta(G)+2$. 
{In particular, this would constitute
a precise version for Cereceda's Conjecture in the case of regular graphs.}

\section{Proving the Main Results for Lists}
\label{main-thms-list:sec}

In this section, we prove the list colouring portions of our main results:
Theorems~\ref{factor2:thm}(i) and~\ref{big-lists:thm}(i).
Recall that a graph $G$ is \emph{factor-critical} if $|V(G)|$ is odd and $G-v$
has a perfect matching for each $v\in V(G)$.
Gallai showed that if every vertex is unsaturated by some maximum matching,
i.e., $\mu(G-v)=\mu(G)$ for every vertex $v$, then $G$ is factor-critical.
So, in each proof we begin with a lemma to handle factor-critical graphs, and
thereafter proceed to the general case.  The following lemma serves as the base
case in an inductive proof of Theorem~\ref{factor2:thm}(i). (We will only need
it when $G$ is factor-critical, but we prove it for all $G$.)

\begin{lem}
\label{lem:lists:d+2}
Let $G$ be graph and let $L$ be a list-assignment for $G$ such that
$|L(v)|\ge \deg(v)+2$ for all $v\in V(G)$.  If $\a$ and $\b$ are
$L$-colourings of $G$, then we can recolour $G$ from $\a$ to $\b$ in at
most $2n(G)-1$ steps.
\end{lem}

\begin{proof}
We use induction on $|\b(V(G))|$.
The case $|\b(V(G))|=1$, is trivial, since then $G$ is an independent set;
starting from $\a$, we can recolour each vertex $v$ to $\b(v)$.
We use at most $n(G)$ recolouring steps, which suffices since $n(G)\le 2n(G)-1$.

Now assume $|\b(V(G))|\ge 2$.  For each colour $c$, let 
$\a^{-1}(c):=\{v\in V(G)\mbox{ s.t. }\a(v)=c\}$ and
$\b^{-1}(c):=\{v\in V(G)\mbox{ s.t. }\b(v)=c\}$.  Since $\sum_{c\in
\a(V(G))}|\a^{-1}(c)|= n(G)=\sum_{c\in \b(V(G))}|\b^{-1}(c)|$, there
exists $c$ such that $|\a^{-1}(c)|\le |\b^{-1}(c)|$.
Starting from $\a$, for each $v\in \a^{-1}(c)$, recolour $v$ to avoid
colour $c$.  Now for each $v\in \b^{-1}(c)$, recolour $v$ with $c$. 
After at most $2|\b^{-1}(c)|$ steps, we have $|\b^{-1}(c)|$ more
vertices that are coloured to agree with $\b$.  Now we delete these
$|\b^{-1}(c)|$ vertices, and delete the colour $c$ from $L(v)$ whenever
$v$ had some neighbours in $\b^{-1}(c)$, and finish by induction.
\end{proof}

\begin{corollary}
\label{cor:lists:d+2}
For every graph $G$ and every $k \ge \Delta(G) + 2$, the
diameter of $\C_k(G)$ is at most $2n(G)-1$.
\end{corollary}

Now we prove Theorem~\ref{factor2:thm}(i).  For easy reference, we restate it
below.

\begin{theorem}
\label{thm:lists:d+2}
Let $G$ be an $n$-vertex graph and $L$ be a list-assignment with
$|L(v)|\ge \deg(v)+2$ for all $v\in V(G)$.  If $\a$ and $\b$ are
$L$-colourings of $G$, then we can recolour $G$ from $\a$ to $\b$ in at
most $n(G)+2\mu(G)$ steps.
\end{theorem}
\begin{proof}
We use induction on $n(G)$.  We assume $G$ is connected;
otherwise, we handle each component separately.  (This suffices because, for a
graph $G$ with components $G_1,\ldots,G_s$, we have $n(G)=\sum_{i=1}^sn(G_i)$
and $\mu(G)=\sum_{i=1}^s\mu(G_i)$.)
If $\mu(G-v)=\mu(G)$ for all $v\in V(G)$, then $G$ is factor-critical.
This is an easy result of Gallai, but also follows from
Theorem~\ref{EG-decomp}(2).  Now $n(G)+2\mu(G)=2n(G)-1$, so we are
done by Lemma~\ref{lem:lists:d+2}.  

Assume instead that some vertex $v$ is saturated by every maximum matching.
Note that $2\deg(v) < 2|L(v)|$. By Pigeonhole there exists
$c\in L(v)$ such that $|\b^{-1}(c)\cap N(v)|+|\a^{-1}(c)\cap
N(v)|\le 1$.  (We handle
the case when equality holds, since the other case is easier.)  Assume there
exists $w\in N(v)$ such that $\a(w)=c$ and $c\notin \cup_{x\in
N(v)\setminus\{w\}} \{\a(x),\b(x)\}$.  (If instead $\b(w)=c$, then we
swap the roles of $\a$ and $\b$.)  Form $\tilde{\a}$ from $\a$ by recolouring
$w$ from $L(w)\setminus (\a(N(w))\cup \{c\})$ and recolouring $v$ with $c$.  Let
$G':=G-v$.  Let
$L'(x):=L(x)-c$ for all $x\in N(v)$ and $L'(x):=L(x)$ for all other $x\in
V(G')$.  Denote the restrictions to $G'$ of $\tilde{\a}$ and $\b$ by $\tilde{\a}'$ and
$\b'$.  By induction, we can recolour $G'$ from $\tilde{\a}'$ to
$\b'$ in at most $n(G')+2\mu(G')$ steps. After this, we can finish by
recolouring $v$ with $\b(v)$. Thus,
$\dist_L(\a,\b)\le 2+\dist_{L'}(\tilde{\a}',\b')+1 \le
2+n(G')+2\mu(G')+1 = 2+n(G)-1+2(\mu(G)-1)+1=n(G)+2\mu(G)$.
\end{proof}


To prove Theorem~\ref{big-lists:thm}(i), we will use the following lemma to
handle the case when $G$ is factor-critical.

\begin{lem}
\label{lem:lists:2d+1}
Let $G$ be a graph and let $L$ be a list-assignment for $G$ such that
$|L(v)|\ge 2\deg(v)+1$ for all $v\in V(G)$.  If $\a$ and $\b$ are $L$-colourings
of $G$, then we can recolour $G$ from $\a$ to $\b$ in at most $\left \lfloor
\frac{3n(G)}2\right \rfloor$ steps.
\end{lem}

\begin{proof}
Let $\a$ and $\b$ be arbitrary $L$-colourings of $G$.
(Recall that all colours are positive integers.)
Let $V_1$ and $V_2$ be the subsets of $V(G)$, respectively,
for which $\a(v)> \b(v)$ and $\a(v)< \b(v)$.  Let $n_1:=|V_1|$ and $n_2:=|V_2|$.
Since $n_1+n_2 \le n(G)$, we assume by symmetry that $n_1 \le \frac {n(G)}2$;
otherwise,
we interchange $\a$ and $\b$.
We show how to recolour $G$ from $\a$ to $\b$ using at most $n(G)+n_1 \le
n(G) + \left \lfloor \frac{n(G)}2\right \rfloor$ steps. 

First, iteratively for every $v \in V_1$ we recolour $v$ from $L(v)\setminus
\cup_{w\in N(v)}\{\b(w),\c(w)\}$, where $\c$ is the current
colouring of $G$.  By definition this recolouring is proper (we actually do not
need to recolour $v$ if $\c(v) \not= \b(w)$ for
all $w \in N(v)$).  Next, iteratively for $j$ from $\max\{\cup_{v \in V}
L(v)\}$ to $\min\{\cup_{v \in V} L(v)\}$, we recolour with $j$ all vertices $v
\in V(G)$ for which $\b(v)=j$.  In these steps, we only have proper colourings, as we will now show. Consider a neighbour $w$ of $v$.  If $w$ has already been recoloured with some
$\gamma(w)$, then $\gamma(w)\ne \b(v)$, by construction.  If $w$ has not been
recoloured, then $\a(w)\le \b(w)<\b(v)$.
\end{proof}


The following classical result is helpful in our proof of
Theorem~\ref{big-lists:thm}(i).

\begin{theorem} [Edmonds--Gallai Decomposition Theorem]
\label{EG-decomp}
For a graph $G$, let
\begin{enumerate}
\item[~~~~] $V_1(G):=\{v\in V(G)$ such that some maximum matching avoids $v\}$,
\item[~~~~] $V_2(G):=\{v\in V(G)$ such that $v$ has a neighbour in $V_1(G)$, but
$v\notin V_1(G)\}$,
\item[~~~~] $V_3(G):=V(G)\setminus (V_1(G)\cup V_2(G))$.
\end{enumerate}
\noindent
Now the following statements hold.\footnote{We prefer the names $V_1(G)$,
$V_2(G)$, $V_3(G)$ over the standard terminology $D(G)$, $A(G)$, and $C(G)$,
since the latter terms conflict with our use of $D=(V,A)$ for a digraph $D$ with
vertex set $V$ and arc set $A$.}
\begin{enumerate}
\item The subgraph induced by $V_3(G)$ has a perfect matching.
\item The components of the subgraph induced by $V_1(G)$ are factor-critical.
\item If $M$ is any maximum matching, then $M$ contains a perfect matching of
each component of $V_3(G)$, and $M$ contains a near-perfect matching of each
component of $V_1(G)$, and $M$ matches all vertices of $V_2(G)$ with vertices
in distinct components of $V_1(G)$.
\item $\mu(G)=\frac12\left(|V(G)|-c(V_1(G))+|V_2(G)|\right)$, where $c(V_1(G))$
denotes the number of components of the graph spanned by $V_1(G)$.
\end{enumerate}
\end{theorem}

Now we prove Theorem~\ref{big-lists:thm}(i).  For easy reference, we restate it
below. 

\begin{theorem}
\label{thr:proofmainconj_2d+1}
Let $G$ be a graph and $L$ be a list-assignment for $G$ with $|L(v)|\ge
2\deg(v)+1$ for all $v\in V(G)$.  If $\alpha$ and $\beta$ are $L$-colourings of
$G$, then we can recolour $G$ from $\alpha$ to $\beta$ in at most $n(G)+\mu(G)$ steps.
\end{theorem}
\begin{proof}
We refer to $V_1(G)$, $V_2(G)$, and $V_3(G)$, as defined in Theorem~\ref{EG-decomp}.
Starting from $\alpha$, recolour each vertex $v\in V_2(G)$ to avoid each colour used
currently on $N(v)$ and also to avoid each colour used on $N(v)$ in $\b$;
call the resulting colouring $\tilde{\a}(G)$.  Let $G':=G-V_2(G)$.
For each $v\in V(G')$, let $L'(v):=L(v)\setminus (\cup_{w\in N(v)\cap
V_2(G)}\tilde{\a}(w))$.  Note that $|L'(v)|\ge 2\deg_{G'}(v)+1$.
Thus, for each component $H$ of $G'$, we can recolour $H$ from $\alpha$
to $\beta$ in at most $\lfloor 3n(H)/2\rfloor$ steps, by Lemma~\ref{lem:lists:2d+1}.
Finally, we recolour each vertex of $V_2(G)$ to its colouring $\b$.
Let $\comp(G')$ denote the set of components of $G'$.
The number of recolouring steps used is at most $2|V_2(G)|+\sum_{H\in \comp(G')}\lfloor
3n(H)/2\rfloor = 2|V_2(G)|+\frac32(|V(G)|-|V_2(G)|)-\frac12c(V_1(G)) =
|V(G)|+\frac12(|V(G)|-c(V_1(G))+|V_2(G)|) = |V(G)|+\mu(G)$.
The final equality uses Theorem~\ref{EG-decomp}(4).
The first equality uses that the vertices of a component $H \in \comp(G')$ are
contained either in $V_1(G)$ or in $V_3(G)$. If $V(H) \subseteq V_3(G)$ then
$n(H)$ is even by Theorem~\ref{EG-decomp}(3). On the other hand, if $V(H)
\subseteq V_1(G)$ then $n(H)$ is odd  by Theorem~\ref{EG-decomp}(2).
\end{proof}

By the comment after Proposition~\ref{lower-bounds:prop},
Theorem~\ref{thr:proofmainconj_2d+1} is sharp.
For (non-list) colouring, we get the following.

\begin{cor}
    For every graph $G$ and every $k \ge 2\Delta(G) +1$, we have $\diam \C_k(G)=n(G)+\mu(G)$.
\end{cor}
\begin{proof}
The upper and lower bounds follow, respectively, from Theorem~\ref{thr:proofmainconj_2d+1} and 
Proposition~\ref{lower-bound-prop}.
\end{proof}


\section{Proving the Main Results for Correspondences}
\label{main-thms-corr:sec}

In this section, we prove the correspondence colouring portions of our main results:
Theorems~\ref{factor2:thm}(ii) and~\ref{big-lists:thm}(ii).

\begin{theorem}
\label{thr:proofmainconj_2d+1_DP}
Let $G$ be a graph and $(L,H)$ be a correspondence cover with $|L(v)|\ge
2\deg(v)+1$ for all $v\in V(G)$.  If $\a$ and $\b$ are $(L,H)$-colourings of $G$, then $G$ can
be recoloured from $\alpha$ to $\beta$ in at most $n(G)+\tau(G)$ steps.
\end{theorem}
\begin{proof}
Let $S$ be a minimum vertex cover.  For every vertex $v\in S$, we recolour $v$
with a colour that, for every $w \in N(v)$, conflicts (under cover $(L,H)$) 
with neither the current colour $\c(w)$ nor the final colour $\b(w)$.
This is possible because $|L(v)|\ge 2\deg(v)+1$.
Now we recolour every vertex $v \in V(G) \backslash S$ with $\b(v)$ and then
do this also for every vertex in $S$.
Note that we use at most $n(G)+\tau(G)$ recolourings.
\end{proof}

\begin{theorem}
\label{thm:lists:d+2_DP}
Let $G$ be a graph and $(L,H)$ be a correspondence cover for $G$ such that
$|L(v)|\ge \deg(v)+2$ for all $v\in V(G)$.  If $\a$ and $\b$ are 
$(L,H)$-colourings of $G$, then we can recolour $G$ from $\a$ to $\b$ in at
most $n(G)+2\tau(G)$ steps.
\end{theorem}
\begin{proof}
We use induction on $\tau(G)$.  
If $\tau(G)=0$, i.e., $G$ is an independent set, then we simply recolour every
vertex $v$ to $\b(v)$.
So assume the theorem is true for all graphs $G'$ with $0\le
\tau(G')<\tau(G)$.

Let $S$ be a minimum vertex cover of $G$ and pick $v \in S$.
%
Consider the union over all $w\in N(v)$ of the edges in $H$ from $\a(w)$ and
$\b(w)$ to $L(v)$.  By Pigeonhole, since $2\deg(v) < 2|L(v)|$,
some colour $c\in L(v)$ has at most one incident edge in this union.
(We handle the case when $c$ has exactly one incident edge, since the other case is
easier.)  Assume there exists $w_0\in N(v)$ such that $\a(w_0)$ is matched to $c$
and no other $\a(w)$ or $\b(w)$ is matched to $c$.  (If instead $\b(w_0)$ is
matched to $c$, then we swap the roles of $\a$ and $\b$.)  Form $\tilde{\a}$
from $\a$ by first recolouring $w_0$ 
with a colour different from $\a(w_0)$, that still
gives a proper $(L,H)$-colouring, and afterward recolouring $v$ with $c$.  Let
$G':=G-v$. For every $w\in N(v)$, remove from $L(w)$ the
colour $c'$ matched with $c$; call the resulting correspondence cover
$(L',H')$.
Denote the restrictions to $G'$ of $\tilde{\a}$ and $\b$ by $\tilde{\a}'$ and
$\b'$.  By induction, since $\tau(G')<\tau(G)$, we can recolour $G'$ from $\tilde{\a}'$ to
$\b'$ in at most $n(G')+2\tau(G')$ steps. After this, we can finish by
recolouring $v$ with $\b(v)$. Thus,
$\dist_{(L,H)}(\a,\b)\le 2+\dist_{(L',H')}(\tilde{\a}',\b')+1 \le
2+n(G')+2\tau(G')+1 = 2+n(G)-1+2(\tau(G)-1)+1=n(G)+2\tau(G)$.
This proves the theorem.
\end{proof}

\section{Distances in Reconfiguration Graphs of Trees}
\label{trees:sec}

For each tree $T$, it is straightforward to check that reconfiguration of correspondence
colouring is no harder than reconfiguration of list colouring.  Given a tree $T$
and a correspondence cover $(L,H)$, it is easy to construct a list-assignment $L'$
for $T$ such that $L'$-colourings of $T$ are in bijection with
$(L,H)$-colourings of $T$.  We can do this by induction on $n(G)$, by
deleting a leaf $v$ and extending the list-assignment $L'$, given by
hypothesis, for $T-v$.  So, for simplicity, we phrase all results in this
section only in terms of list colourings.

\begin{theorem}
\label{thr:n+tau_forpath}
Fix a tree $T$ and a list-assignment $L$ with $\lvert L(v) \rvert \ge
\deg(v)+2$ for every $v \in V(T).$ If $\a$ and $\b$ are proper $L$-colourings
of $T$, then the distance between $\a$ and $\b$ in the reconfiguration graph
$\C_L(T)$ is equal to $\mu(D_{\a, \b})+\sum_{v \in V} \mathbf{1}_{\a(v) \not= \b(v)}.$
\end{theorem}

\begin{proof}
The lower bound holds by Observation~\ref{lower-bound:obs}.
Now we prove that this lower bound is also an upper bound.  
 
We use induction on $n(T)$.  The base case, $n(T)=1$, is trivial.  
Assume the theorem is true whenever $T$ has order at most $s-1$.
We will prove it for an arbitrary tree $T$ on $s$ vertices.
If $\a(v)=\b(v)$ for some $v \in V(T)$, then we use induction on the components
of $T-v$, with $\a(v)$ removed from the lists of the neighbours
of $v$.
    
Suppose instead there are neighbours $v$ and $w$ for which $\vc{vw} \not \in A(D_{\a, \b})$.
Now $T-vw$ has two components; so let $C_v$ and $C_w$ be the
components, respectively, containing $v$ and $w$.
By deleting $\a(w)$ from $L(v)$, we can first recolour $C_v$.
Afterwards, we delete $\b(v)$ from $L(w)$ and recolour $C_w$.
By using the induction hypothesis (twice), we get the desired upper bound, since
$\mu(D_{\a, \b})=\mu(D_{\a, \b}[C_v])+\mu(D_{\a, \b}[C_w]).$
    
In the remaining case, we have two colour
classes, say with colours $1$ and $2$, which need to be swapped. 
We pick a smallest vertex cover $S$, which has size $\tau(D_{\a,\b})=\mu(D_{\a,\b}).$
Iteratively, we recolour every vertex $v$ in $S$ with a colour different from
$1$, $2$, and all colours used to recolour neighbours of $v$ in $S$.
Note that $v$ has at most $\deg(v)-1$ neighbours in $S$, since otherwise $S$
would not be a minimal vertex cover.
Since $\lvert L(v) \rvert \ge \deg(v)+2$, we can recolour as desired.
Next, for each $w \in V(T)\setminus S$, we recolour $w$ with $\b(w).$
Finally, we recolour each $v \in S$ with $\b(v)$.
This proves the induction step, which finishes the proof.
%
\end{proof}

\begin{cor}\label{cor:diamle_n+mu}
    If $T$ is a tree and  $L$ is a list-assignment with $\lvert L(v) \rvert \ge \deg(v)+2$ for every $v \in V(T)$,
    then 
    $$\diam \C_L(T) \le n(T)+\mu(T)\le \lfloor 3n(T)/2\rfloor.$$
\end{cor}
\begin{proof}
    Theorem~\ref{thr:n+tau_forpath} implies that $\diam~\C_L(T)\le
n(T)+\mu(T)$, and the bound $\mu(T) \le \lfloor \frac {n(T)}2 \rfloor$ holds trivially, since
each edge of a matching saturates two vertices.
\end{proof}

Recall that $[k]$ denotes $\{1,\ldots,k\}$.  Thus, we write
\emph{$[k]$-colouring} to mean a $k$-colouring from the colours $[k]$.

\begin{prop}\label{prop:diam&rad_ofCkT}
    For every tree $T$ and $k \ge \Delta(T)+2$,
    we have
    $$\diam \C_k(T) =n(T)+\mu(T) \mbox{ and } \rad \C_k(T) \ge n(T)+\left\lceil
\frac{\mu(T)}{2} \right\rceil.$$
\end{prop}

\begin{proof}
    The inequality $\diam \C_k(T) \le n(T)+\mu(T)$ holds by
Corollary~\ref{cor:diamle_n+mu}; and this inequality actually holds with
equality, since the two $[2]$-colourings of $T$ are at distance exactly $n(T)+\mu(T)$.
That is, $\diam \C_k(T) = n(T)+\mu(T)$.
Now we consider $\rad \C_k(T)$.
    If we contract all edges of a maximum matching $M$ in $T$, the result is
also a tree, $T'$.
    When we 2-colour $T'$,
    one colour is used on at least half of the contracted edges. Denote the set
of these contracted edges by $E'$. Note that $E'$ is an induced matching in $T$.
    Fix an arbitrary proper $k$-colouring $\alpha$ of $T$.
    Now we construct a proper $k$-colouring $\beta$ of $T$ by swapping the
colours on the endpoints of each edge $vw \in E'$, i.e., let $\b(v):=\a(w)$ and
$\b(w):=\a(v)$ for every edge $vw \in E'.$
    For every vertex $v$ not belonging to an edge in $E'$, choose a colour in
$[k]$ different from the colours already assigned (in $\b$) to 
neighbours of $v$ and different from $\a(v).$
    Now $\dist(\a,\b) \ge n(T)+\lvert E' \rvert \ge n(T)+\left\lceil
\frac{\mu(T)}{2} \right\rceil $ by Theorem~\ref{thr:n+tau_forpath}.
\end{proof}


We construct trees $T$ for which certain colourings are more ``central'' in the
reconfiguration graph than others.
That is, we construct trees $T$ for which $\rad\C_k(T) \not= \diam\C_k(T)$.
We also study the maximum possible diameter and minimum possible radius of
reconfiguration graphs of trees of given order, and the maximum possible
difference of these quantities.

\begin{prop}
    For every $k\ge 4$, the path $P_n$ satisfies $$\diam \C_k(P_n) = \left\lfloor \frac{3n}{2} \right\rfloor \mbox{ and } \rad \C_k(P_n) = \left\lceil \frac{4n-1}{3} \right\rceil.$$
    Furthermore, there exist $n$-vertex trees $T$, with maximum degree 3, such that
for every $k\ge 5$ we have $$\diam \C_k(T) = \left\lfloor \frac{3n}{2}
\right\rfloor \mbox{ and } \rad \C_k(T) = \left\lceil \frac{5n-1}{4} \right\rceil.$$
All such $T$ maximise, over all $n$-vertex trees, the difference $\diam
\C_k(T)-\rad\C_k(T)$.
\end{prop}

\begin{proof}
    Consider $P_n$ with vertex set $\{v_1, v_2, \ldots, v_n\}$. Fix $k \ge 4$, and
    let $\a$ and $\b$ be the colourings with ranges $[3]$ and $[2]$, respectively, such that 
    $$\a(v_i) \equiv i \Mod{3} \mbox{~~~ and~~~ } \b(v_i) \equiv i \Mod{2}~~~~~\forall i \in [n].$$
    
Note that $\ecc(\b)=\lfloor \frac{3n}{2} \rfloor$ by
Theorem~\ref{thr:n+tau_forpath}, since switching colours $1$ and
$2$ in $\b$ requires $\left\lfloor \frac{3n}{2} \right\rfloor$ recolouring steps.
Furthermore, Corollary~\ref{cor:diamle_n+mu} implies that $\diam
\C_k(P_n)=\ecc(\b)=\lfloor \frac{3n}2\rfloor$.
On the other hand, for every proper colouring $\gamma$, we know that $D_{\a,
\gamma}$ cannot have a pair of bidirected edges of the form $v_iv_{i+1}$ and
$v_{i+2}v_{i+3},$ since $\a(v_i)=\a(v_{i+3}).$
So $\mu(D_{\a, \gamma})\le \lceil \frac{n-1}{3} \rceil$; hence, $\ecc(\a) \le
n+\lceil \frac{n-1}{3} \rceil$, which implies that $\rad \C_k(P_n)\le\lceil\frac{4n-1}3\rceil$.

Next we prove the lower bound on $\rad\C_k(P_n)$. 
For every colouring $\a$ of $P_n$, we can choose at least
$\lceil \frac{n-1}{3} \rceil$ disjoint edges such that swapping the colours in $\a$ on the
endpoints of each edge (and possibly recolouring vertices not in any of these
edges) yields another proper colouring.
To see this, first select edge $v_1v_2.$
Whenever $v_iv_{i+1}$ has been selected, there exists $j \in \{i+3, i+4\}$
for which $\a(v_i) \not=\a(v_j)$. Now add edge $v_{j-1}v_j$ to the set of selected edges.
When our selection ends, the set $E'$ contains at least $\lceil \frac{n-1}{3} \rceil$ edges.
We can now construct a $[k]$-colouring $\b$ where $\b(w)=\a(v)$ and
$\b(v)=\a(w)$ for every edge $vw \in E'$, and $\b(v) \ne \a(v)$ for every $v \in
V(P_n).$
Theorem~\ref{thr:n+tau_forpath} gives $\dist(\a, \b) \ge n+\lceil \frac{n-1}{3} \rceil$.
 
Now let $T_n$ be the $n$-vertex comb graph; see Figure~\ref{comb-fig}.
Here $T_n$ is formed from a path $v_1v_2\ldots v_t$, where $t=\lceil
\frac n2 \rceil$, by adding, for each $v_i$ (except $v_t$ when $n$ is
odd), an additional neighbour $w_i$.
Since $\mu(T_n)=\left\lfloor\frac{n}{2}\right\rfloor$, 
by Proposition~\ref{prop:diam&rad_ofCkT}
the diameter of $\C_k(T_n)$ is $\lfloor \frac{3n}2\rfloor$.
    Let $\a$ be the colouring shown and described in Figure~\ref{comb-fig}.

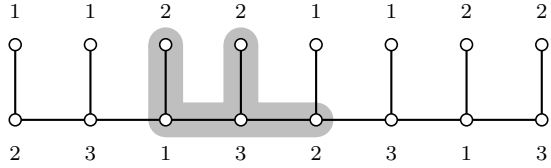
\begin{figure}[!h]
\centering
\begin{tikzpicture}[thick, scale=1]
\tikzstyle{uStyle}=[shape = circle, minimum size = 4.5pt, inner sep = 0pt,
outer sep = 0pt, draw, fill=white, semithick]
\tikzstyle{sStyle}=[shape = rectangle, minimum size = 4.5pt, inner sep = 0pt,
outer sep = 0pt, draw, fill=white, semithick]
\tikzstyle{lStyle}=[shape = circle, minimum size = 4.5pt, inner sep = 0pt,
outer sep = 0pt, draw=none, fill=none]
\tikzset{every node/.style=uStyle}
\def\off{.45}
\def\fillwidth{2.3mm}
\def\fillcolor{black!25!white}

\draw[white] (0,2.1) -- (1,2.1); 

\draw[draw=none, fill=\fillcolor] 
(2,1) circle (\fillwidth)
(2,0) circle (\fillwidth)
(3,1) circle (\fillwidth)
(4,0) circle (\fillwidth)
;
\draw[color=\fillcolor, line width=2*\fillwidth+.1] (2,1) -- (2,0) (2,0) --
(4,0) (3,1) -- (3,0);

\draw (0,0) -- (7,0);
\foreach \i in {0,...,7}
{
\draw (\i,1) node (w\i) {} -- (\i,0) node (v\i) {};
}
\foreach \i/\wlab/\vlab in {0/1/2, 1/1/3, 2/2/1, 3/2/3, 4/1/2, 5/1/3, 6/2/1, 7/2/3}
{
\draw (w\i) ++ (0,\off) node[lStyle] {\footnotesize{\wlab}};
\draw (v\i) ++ (0,-\off) node[lStyle] {\footnotesize{\vlab}};
}

\draw (3,-2.25) node[lStyle, shape=rectangle] {
$\a(v_i)=\left.
  \begin{cases}
    1, & \text{for } i \equiv 3 \Mod 4\\
    2, & \text{for } i \equiv 1 \Mod 4\\
    3, & \text{for } i \equiv 0,2 \Mod 4
  \end{cases}\right.
  \mbox{~~~~} 
\a(w_i)=\left.  
\begin{cases}
    1, & \text{for } i \equiv 1,2 \Mod 4\\
    2, & \text{for } i \equiv 3,4 \Mod 4.
  \end{cases}\right. 
$
};
\end{tikzpicture}

\caption{The comb graph $T_{16}$ with colouring $\a$.  The shaded region
denotes four edges no two of which, for any colouring $\b$, yield
disjoint bidirected edges of $D_{\a,\b}$.\label{comb-fig}}
\end{figure}
For any proper colouring $\b$ of $T_n$, note that 
$\mu(D_{\a,\b}) \le \left\lceil \frac{n-1}{4} \right\rceil$.
This holds because, for each odd $i$, the subgraph
$D_{\a,\b}[\{v_i,v_{i+1},v_{i+2},w_i,w_{i+1}\}]$ contains 
no two disjoint bidirected edges.
Thus, $\ecc(\a) \le \ceil{\frac{5n-1}{4}}$, by
Theorem~\ref{thr:n+tau_forpath}.  Since $\mu(T_n)= \floor{\frac{n}{2}}$, 
Proposition~\ref{prop:diam&rad_ofCkT} implies that
$\rad \C_k(T_n) = n+\ceil{\floor{n/2}/2} = \ceil{
\frac{5n-1}{4} }$.  Thus,
$\diam\C_k(T_n)-\rad\C_k(T_n)=\floor{\frac{3n}2}-
\ceil{\frac{5n-1}4}=\floor{\frac{n}4}$.
For every $n$-vertex tree $T$, 
Proposition~\ref{prop:diam&rad_ofCkT} implies that $\diam\C_k(T)-\rad\C_k(T) \le
(n+\mu(T))-(n+\ceil{\frac{\mu(T)}2})=\floor{\frac{\mu(T)}2}\le\floor{\frac{n}4}$.
Thus, $T_n$ maximises this difference.
\end{proof}

For list colourings, the difference $\diam\C_L(T)-\rad\C_L(T)$ can
be even larger.
\begin{prop}
\label{tree-lower-bound}
    For every tree $T$, there exists a list-assignment $L$,
with $\lvert L(v) \rvert \ge \deg(v)+2$ for all $v \in V(T)$, such that
$\diam \C_L(G)=n(T)+\mu(T)$ and $\rad \C_L(G)=n(T)$.
These values are, respectively, the maximum and minimum possible over all such
list-assignments $L$.
\end{prop}

\begin{proof}
    Choose a maximum matching $M$ of $T$ and fix $L$ such that 
$$\lvert L(v) \cap L(w) \rvert =\left.
  \begin{cases}
    0, & \text{when } vw \not \in M\\
    2, & \text{when } vw \in M.
  \end{cases}\right.$$
  Let $\a$ and $\b$ be $L$-colourings such that, for every edge $vw \in M$, we
have $\a(v)=\b(w)$ and $\a(w)=\b(v)$.
  For every vertex $v$ not saturated by $M$, we pick $\a(v)$ and $\b(v)$ to be
arbitrary distinct colours in $L(v)$.
  By Theorem~\ref{thr:n+tau_forpath}, we have $\dist(\a,\b)=n(T)+\mu(T)$. 
Corollary~\ref{cor:diamle_n+mu} implies that $\diam \C_L(T)=n(T)+\mu(T)$.
Further, this diameter is the maximum possible over all such list-assignments
$L$.

To show that $\rad \C_L(T)\le n(T)$, we construct an $L$-colouring $\c$
such that $\dist(\c,\d)\le n(T)$ for every $L$-colouring $\d$.
Since $|L(v)\setminus \left(\cup_{w\in N(v)}L(w)\right)|\ge 3-2=1$ for every
vertex $v$, we can form an
$L$-colouring $\c$ such that every $v\in V(T)$ satisfies $\c(v) \notin \cup_{w\in
N(v)}L(w)$.  
For every proper colouring $\d$, by our construction of $\c$, 
we can recolour $\d$ into $\c$ greedily; hence, $\dist(\d,\c) \le n(T)$.

Now we show that $\rad C_L(T)\ge n(T)$.
For every $L$-colouring $\c$, there exists an $L$-colouring $\d$ such that
$\d(v)\ne \c(v)$ for all $v$.  To see this, let $L'(v):=L(v)\setminus \c(v)$.
Since $|L'(v)|\ge \deg(v)+1$ for all $v$, we form $\d$ by colouring $G$ greedily
(in any order) from $L'$.  Clearly, $\dist(\c,\d)\ge n(G)$.
Thus, $\rad \C_L(T)=n(T)$.  In fact, this lower bound does not depend on the
specific choice of $L$, but only uses that $|L(v)|\ge \deg(v)+2$ for all $v$.  
Thus, this radius is the minimum over all such list-assignments $L$.
\end{proof}

%
%
\section{Proving the List Conjecture 
for Complete Bipartite Graphs and Cactuses}
\label{classes-list:sec}

A \emph{cactus} is a connected graph in which each edge lies on at most one cycle.
In this section, we prove the \hyperref[conj:main_list]{List Conjecture} 
for all complete bipartite graphs and for all cactuses.  Both proofs are by
induction, and rely on the following helpful lemma.

\begin{lem}
\label{lem:StructureMinimalCounterexample}
Fix a positive integer $b$.
Let $G=(V,E)$ be a graph for which there exists a
list-assignment $L$ and two proper $L$-colourings $\a, \b$
such that $\lvert L(v) \rvert \ge \deg(v)+b$ for every $v \in V(G)$, and
$\dist_L(\a,\b)>n(G)+\mu(G)$, but, for every proper
induced subgraph of $G$, no such list-assignment exists.
Now for every partition $V(G)=V_1 \cup V_2$ into two non-empty subsets of
vertices, $D_{\a,\b}$ has at least one arc from $V_1$ to $V_2$ and at least one
arc from $V_2$ to $V_1,$ i.e., $D_{\a,\b}$ is strongly connected.
\end{lem}

\begin{proof}
Assume the lemma is false; by symmetry, assume
$D_{\a,\b}$ has no arcs from $V_1$ to $V_2$.
Let $G_1:=G[V_1]$ and $G_2:=G[V_2]$.
Let $\c$ be the $L$-colouring where $\c(v):=\b(v)$ when $v \in V_1$ and
$\c(v):=\a(v)$ when $v \in V_2$.
    
For every $v \in V_1$, let $L'(v):=L(v) \backslash \{\a(w) \mid w \in N(v) \cap V_2\}$.
Note that $\lvert L'(v) \rvert \ge \deg_{G_1}(v)+b$. Also note that still $\c(v) \in L'(v)$ for every $v\in V_1$; this is where we use that there is no arc from $V_1$ to $V_2$.
Since $G_1$ is a proper induced subgraph of $G$, by hypothesis
$\diam \C_{L'}(G_1) \le n(G_1)+\mu(G_1)$;  
thus $\dist_L(\a,\c) \le n(G_1)+\mu(G_1)$.
Now for every $v \in V_2$, let $L'(v):=L(v) \backslash \{\b(w) \mid w \in N(v) \cap V_1\}$.
Similarly to before, $\lvert L'(v) \rvert \ge \deg_{G_2}(v)+b$, so $\diam
\C_{L'}(G_2) \le n(G_2)+\mu(G_2)$; thus $\dist_L(\c,\b) \le n(G_2)+\mu(G_2)$.
By the bounds above and the triangle inequality, we have
    \begin{align*}
        \dist_L(\a,\b) &\le \dist_L(\a,\c)+\dist_L(\c,\b) \\
        &\le n(G_1)+\mu(G_1)+n(G_2)+\mu(G_2)\\
        &\le n(G)+\mu(G).
    \end{align*}
The last step uses that the union of a matching in $G_1$ and a matching
in $G_2$ is a matching in $G$. 
So $G$ is not a counterexample, and the lemma is true.
\end{proof}

It is helpful to note that an analogous statement holds for correspondence
colouring.  In fact, its proof is nearly identical to that given above.
We use this observation in our proof of Theorem~\ref{thr:subcubic_cover}.

Using Lemma~\ref{lem:StructureMinimalCounterexample},
we prove that the 
\hyperref[conj:main_list]{List Conjecture} 
holds for all complete bipartite graphs.

\begin{theorem}
\label{Kpq:thm}
The 
\hyperref[conj:main_list]{List Conjecture} 
is true for all complete bipartite graphs. 
\end{theorem}
\begin{proof}
Suppose the theorem is false and choose a counterexample $K_{p,q}$ minimizing
$p+q$.  By symmetry, we assume that $p\le q$.  Denote the parts of $G$ by $U$
and $W$, with $|U|=p$ and $|W|=q$.

By Lemma~\ref{lem:StructureMinimalCounterexample}, we have $\a(U)\subseteq\b(W)$ and
$\a(W)\subseteq\b(U)$.  By swapping the roles of $\a$ and $\b$, we also have
$\b(U)\subseteq\a(W)$ and $\b(W)\subseteq\a(U)$; thus, $\a(U)=\b(W)$ and
$\a(W)=\b(U)$.
Note that $\a(U)\cap\b(U)=\a(U)\cap \a(W)=\emptyset$, because the graph is complete bipartite.

\textbf{Case 1: $\bm{|W|\ge |\a(U)|+|\b(U)|-1}$.}
For each $u\in U$, we have $|L(u)|\ge |W|+2\ge
|\a(U)|+|\b(U)|+1=|\b(W)|+|\a(W)|+1$.  First recolour each $u\in U$ from
$L(u)\setminus(\a(W)\cup \b(W))$.  Now recolour each $w\in W$ with $\b(w)$.
Finally, recolour each $u\in U$ with $\b(u)$.  The number of steps that we use is
at most $2|U|+|W|=n(G)+\mu(G)$.

\textbf{Case 2: $\bm{|W|\le |\a(U)|+|\b(U)|-2}$.}
If $|W|\ge 2|\a(U)|$ and $|W|\ge 2|\b(U)|$, then $2|W|\ge 2|\a(U)|+2|\b(U)|$,
which contradicts the case.  So assume, by symmetry, that $|W|\le 2|\b(U)|-1$;
if not, then simply interchange the roles of $\a$ and $\b$.  Now, by Pigeonhole,
there exists $c\in \b(U)$ such that $|\a^{-1}(c)|= |\a^{-1}(c)\cap W|\le 1$.  So, recolour
$w\in \a^{-1}(c)$, if such $w$ exists, to avoid $\a(U)\cup\{c\}$, and then
recolour every $u\in \b^{-1}(c)$ with $c$.  Now we delete every $u$ such that $\b(u)=c$, we delete
$c$ from $L(w)$ for every $w\in W$, and we finish on the resulting smaller graph $G_2$
(with an assignment of smaller lists) by the minimality of $G$. In total, the number
of recolouring steps we use is at most $|\a^{-1}(c)| + |\b^{-1}(c)| + |V(G_2)| + \mu(G_2) \le  
|\a^{-1}(c)| + |\b^{-1}(c)| + ( n(G)- |\b^{-1}(c)| ) + ( \mu(G) - |\b^{-1}(c)| ) \leq n(G) +\mu(G)$.
\end{proof}

Using Lemma~\ref{lem:StructureMinimalCounterexample},
we prove that the 
\hyperref[conj:main_list]{List Conjecture} 
holds for all cycles.

\begin{lem}
\label{prop:cyclessatisfymainconj}
Let $G$ be a cycle $v_1v_2\cdots v_n$ and
$L$ be a list-assignment such that, for all $v_i$, we have $\lvert L(v_i) \rvert
\ge 4$.  If $\a$ and $\b$ are proper $L$-colourings of $G$, 
then we can recolour $G$ from $\a$ to $\b$ in at most $\lfloor 3n/2\rfloor$ steps.
\end{lem}
\begin{proof}
By Lemma~\ref{lem:StructureMinimalCounterexample}, with $b=2$, if $\dist(\a, \b)>
\lfloor 3n/2\rfloor$, then $D_{\a,\b}$ is strongly connected.
This implies that $D_{\a,\b}$ contains a directed cycle $C_n$ as a subdigraph.
(If $D_{\a,\b}$ contains a bidirected path $P$, then $|\a(V(P))\cup \b(V(P))|=2$.
So, if $P$ is spanning, then its order must be even, to
avoid a conflict between the colours of its endpoints.  But then $D_{\a,\b}$
contains an additional arc, so $D_{\a,\b}$ contains a directed cycle, as claimed.)

If $D_{\a,\b}$ is a bidirected cycle, then $n$ must be even, as in the previous
paragraph, and $\frac{3n}{2}$ recolouring steps suffice.
Suppose instead that $D_{\a,\b}$ is precisely a directed cyle. 
When $n=3$, we recolour one vertex $v$ in a colour absent from
$\a(V(G))\cup\b(V(G))$, recolour the other two vertices in order (to match
$\b$), and finally recolour $v$ with $\b(v)$.  When $n\ge 4$, we recolour one
vertex $v$ in a colour absent from $\a(v)\cup \a(N(v))$.
Now we can recolour correctly (in order) all vertices of $G$ except for $v$ and one
of its neighbours.  Correctly colouring these final two vertices takes at most 3
recolouring steps. Now we are done, since $1+(n-2)+3 \le \lfloor 3n/2\rfloor$.

In the remaining case, we have $n \ge 4$ and some directed edge is adjacent to a
bidirected edge.  Without loss of generality, we assume $\vc{v_1v_2},
\vc{v_2v_3}, \vc{v_2v_1}\in A(D_{\a,\b})$ but $\vc{v_3v_2}\notin A(D_{\a,\b})$.
Recolour $v_1$ with a colour $c$ different from $\a(v_1)$, $\a(v_2)$, and $\a(v_n)$.
Now delete $c$ from $L(v_2)$ and $L(v_n)$, and let $G':=G-v_1$.
By Theorem~\ref{thr:n+tau_forpath}, 
we can recolour $G'$ from $\a$ to $\b$ (both restricted to $G'$) 
using at most $(n-1)+\left \lfloor \frac{n-2}{2}\right \rfloor$ 
recolouring steps. 
Here we use that $\vc{v_3v_2}$ is not an arc, so $\mu(D_{\a,\b}[V(G')])
\le \left \lfloor \frac{n-2}{2}\right \rfloor$.
Finally, we recolour $v_1$ with $\b(v_1)$.
\end{proof}

Using 
Lemmas~\ref{lem:StructureMinimalCounterexample} 
and~\ref{prop:cyclessatisfymainconj},
we can prove that the 
\hyperref[conj:main_list]{List Conjecture} 
is true for every cactus. 


\begin{theorem}
The \hyperref[conj:main_list]{List Conjecture} 
is true for every cactus.
\end{theorem}

\begin{proof}
Instead assume the theorem is false.  Let $G$ be a counterexample
minimizing $n(G)$ and let $\a$ and $\b$ be $L$-colourings with $\dist(\a, \b)>n(G)+\mu(G)$.
Every proper induced subgraph $H$ of $G$ is a disjoint union of cactuses, say
$H_1,\ldots,H_r$;
since $n(H)=\sum_{i=1}^rn(H_i)$ and $\mu(H)=\sum_{i=1}^r\mu(H_i)$, 
the \hyperref[conj:main_list]{List Conjecture} 
must hold for $H$, by the minimality of $G$.
Lemma~\ref{lem:StructureMinimalCounterexample} implies that $D_{\a,\b}$ is strongly connected.
And Lemma~\ref{prop:cyclessatisfymainconj} implies that $G$ is not a cycle.
The following claim restricts the structure of $G$, $\a$, and $\b$.

\begin{claim}
\label{key-helper}
If $v\in V(G)$ and $\mu(G-v)=\mu(G)-1$, then $L(v)\subseteq \cup_{w\in
N(v)}\{\a(w),\b(w)\}$.  In particular, each block of $G$ containing $v$ gives
rise to exactly two arcs in $D_{\a,\b}$ incident to $v$, and $v$ has no adjacent
leaf in $G$.
\end{claim}
\begin{claimproof}
Instead assume there exists $c\in L(v)\setminus\cup_{w\in
N(v)}\{\a(w),\b(w)\}$.  Recolour $v$ with $c$, let $G':=G-v$, let
$L'(w):=L(w)\setminus\{c\}$ for all $w\in N(v)$, and otherwise let
$L'(w):=L(w)$.  
Denote by $\a'$ and $\b'$ the restrictions to $G'$ of $\a$ and $\b$.
Since $G$ is minimal, we can recolour $G'$ from $\a'$ to
$\b'$, using $L'$, in at most $n(G')+\mu(G')=n(G)-1+\mu(G)-1$
recolouring steps.  We finish by recolouring $v$ with $\beta(v)$.
This proves the first statement.

Since $D_{\a,\b}$ is strongly connected, each block $B$ of $G$ containing $v$
gives rise to at least two arcs in $D_{\a,\b}$ incident to $v$, one in each
direction.  Since $G$ is a cactus, each such $B$ contains at most 2 neighbours
of $v$.  So at least half of the neighbours of $v$ are coloured $\b(v)$ by
$\a$, and also at least half of the neighbours of $v$ are coloured $\a(v)$ by
$\b$. Therefore $|\cup_{w\in N(v)}\{\a(w),\b(w)\}|\le 2\cdot (\deg(v)/2+1))=
\deg(v)+2$.  If any such $B$ is either $K_2$ or gives rise to at least three
arcs in $D_{\a,\b}$ incident with $v$, then $|\cup_{w\in
N(v)}\{\a(w),\b(w)\}|\le \deg(v)+1<\lvert L(v)\rvert$, which contradicts the
first statement.   This proves the second statement.  
\end{claimproof}

    
It is easy to check that every neighbour $v$ of a leaf $w$ in $G$ satisfies
$\mu(G-v)=\mu(G)-1$.  Thus, Claim~\ref{key-helper} implies that $G$ has no leaf
vertex.  This implies that all endblocks of $G$ are cycles.
Let $C_s$ be an endblock.  
Denote its vertices by $w_1,\ldots,w_s$, such that $w_s$ is the unique cut-vertex in the block.

\begin{claim}
The endblock $C_s$ is not an even cycle.
\label{even-cycle-clm}
\end{claim}
\begin{claimproof}
Assume instead that $C_s$ is an even cycle.
It is
easy to check that $\mu(G-w_i)=\mu(G)-1$ whenever $i$ is even.
That is, every maximum matching saturates $w_i$ whenever $i$ is even.  By
Claim~\ref{key-helper}, every $w_i$ is incident in $D_{\a,\b}$ to exactly two
arcs arising from $C_s$.  Since $D_{\a,\b}$ is strongly connected, this implies
that $D_{\a,\b}[V(C_s)]$ is a directed cycle; by symmetry, we assume that it is
oriented as $\vc{w_sw_1},\vc{w_1w_2},\ldots,\vc{w_{s-1}w_s}$.

We first recolour $w_s$ with a colour absent from $\a(N[w_s])=\left(\cup_{x\in
N(w_s)} \a(x) \right)\cup \b(w_{s-1})$.
Now we recolour each $w_i$ to
$\b(w_i)$, with $i$ decreasing from $s-1$ to $2$.  Let
$G':=G-\{w_1,\ldots,w_{s-1}\}$, let
$L'(w_s):=L(w_s)\setminus\{\a(w_1),\b(w_{s-1})\}$, and otherwise let
$L'(v):=L(v)$.  Since $G$ is minimal, we can recolour $G'$ from its current
colouring to $\b$ (restricted to $G'$) using at most $n(G')+\mu(G')\le
(n(G)-(s-1))+(\mu(G)-\frac{s}2+1)$ steps.  Finally, recolour $w_1$ to $\b(w_1)$.  The
total number of recolouring steps is at most
$1+(s-2)+(n(G)-(s-1)+\mu(G)-\frac{s}2+1)+1 = n(G)+\mu(G)+2-\frac{s}2$.
Since $s\ge 4$, this is at most $n(G)+\mu(G)$. 
\end{claimproof}
    
     \begin{claim}
        The endblock $C_s$ is not an odd cycle.
        \label{odd-cycle-clm}
    \end{claim}
    \begin{claimproof}
Assume instead that $C_s$ is an odd cycle.
        If $D_{\a,\b}[C_s - w_s]$ contains fewer than
$\floor{\frac{s}{2}}$ disjoint digons (for example, this is true when $s=3$),
then we can easily finish, as follows.
We recolour $w_s$ in a colour different from $\a(N(w_s)) \cup \b(\{w_1,w_{s-1}\})$.
This is possible because $\b(w_s)$ is used by $\a$ on at least two neighbours of
$w_s$, since $D_{\a,\b}$ is strongly connected.
By Theorem~\ref{thr:n+tau_forpath},
we recolour $C_s - w_s$ to $\beta$ in at most
$(s-1)+\floor{\frac{s}{2}}-1$ steps, 
and we then recolour $G \backslash (C_s-w_s)$ to $\b$ in at most
$n(G)-(s-1)+\mu(G)-\floor{\frac{s}{2}}$ steps.  Thus, we recolour $G$ from $\a$
to $\b$ in at most $n(G)+\mu(G)$ steps.

Now we consider the other case, when 
$w_iw_{i+1}$ is a digon of $D_{\a,\b}$ for every odd $i$ with $i<s$.
Since $D_{\a,\b}[C_s]$ is strongly connected, we assume $D_{\a,\b}$
contains arc $w_iw_{i+1}$ for every $0\le i \le s-1$.
Similar to the proof of Lemma~\ref{prop:cyclessatisfymainconj},
here $D_{\a,\b}[C_s]$ cannot contain a spanning bidirected path, since $s$ is odd.
%
%
This implies, for some even $j$, that $D_{\a,\b}$ does not contain $w_{j+1}w_j$.
Now we will recolour some set of $w_i$'s to reach a new colouring $\tilde{\a}$,
such that $D_{\tilde{\a},\b}[C_s]$ is either acyclic or contains a single directed
cycle, $w_{s-2}w_{s-1}$.  From $\tilde{\a}$, we can recolour $G\setminus
(C_s-w_s)$ by the minimality of $G$, and afterward finish on $C_s-w_s$.  The
details follow.

First suppose that $\b(w_s)=\a(w_{s-1})$.
For every $i$ such that either (a) $i < j$ and $i$ is odd or (b) $i > j$ and $i$
is even, recolour $w_i$ with a colour different from those in $\{\a(w_{i-1}),
\a(w_i),\a(w_{i+1})\}$. Here we take the indices modulo $s$, i.e., $w_0=w_s$.
If $\b(w_s)\ne \a(w_{s-1})$, then we recolour the same set of $w_i$'s, except
for $w_{s-1}$.  Let $G':=G-\{w_1,\ldots,w_{s-1}\}$.
By the minimality of $G$, we can recolour $G'$ in at most
$n(G')+\mu(G')=(n(G)-(s-1))+(\mu(G)-\floor{\frac{s}{2}})$ steps. 
%
In the case that $\b(w_s)\ne \a(w_{s-1})$, we now recolour $w_{s-1}$ 
with a colour different from those in $\{\a(w_{s-2}), \a(w_{s-1}),\b(w_{s})\}$. 
For every odd $i$ such that $j+1 \le i \le s-2$, we recolour $w_i$ with $\b(w_i)$.
Next, for every even $i$ such that $2 \le i \le j$, we recolour $w_i$ with $\b(w_i).$
Finally, the remaining vertices in $C_s$ can also be recoloured with their colour in $\b$.
This process uses at most $n(G)+\mu(G)$ steps.
    \end{claimproof}
    
Recall that every endblock of $G$ is a cycle, as observed following
Claim~\ref{key-helper}.  Thus, Claims~\ref{even-cycle-clm}
and~\ref{odd-cycle-clm} yield a contradiction, which proves the theorem.
\end{proof}

\section{Proving the 
Correspondence Conjecture 
for Cactuses, Subcubic
Graphs, and Graphs with Low Maximum Average Degree}
\label{classes-correspondence:sec}

The \emph{maximum average degree} of a graph $G$, denoted $\mad(G)$, is the
maximum, taken over all subgraphs $H$, of the average degree of $H$.  That is,
$\mad(G):=\max_{H\subseteq G}2|E(H)|/|V(H)|$.  Let $\deg^1(v)$ denote the number
of neighbours $w$ of $v$ such that $\deg(w)=1$.
In this section we prove the 
\hyperref[conj:main_DP]{Correspondence Conjecture} 
for all subcubic graphs, cactuses, and graphs $G$ with $\mad(G)<2.4$. 
To do so, we often implicitly use the following
observation.  We omit its proof, which is easy.

\begin{obs}
If $G$ is a graph with a minimum vertex cover $S$ and $v\in S$, then
$\tau(G-v)=\tau(G)-1$.
\end{obs}

\begin{theorem}
\label{thr:subcubic_cover}
Let $G$ be a graph and $(L,H)$ a correspondence cover for $G$ such that, for all
$v\in V(G)$ we have $|L(v)|\ge \deg(v)+2$. Now $\diam\C_{(L,H)}(G)\le
n(G)+\tau(G)$ if at least one of the following holds.
\begin{itemize}
\item[(a)] $\Delta(G)\le 3$.
\item[(b)] $G$ is a cactus.
\item[(c)] $\mad(G)<2.4$.
\end{itemize}
\end{theorem}


\begin{proof}
Fix $G$ satisfying (a), (b), or (c) and $(L,H)$ as in the theorem.  Fix
arbitrary $(L,H)$-colourings $\a$
and $\b$.  Our proof is by a double induction: primarily on $\tau(G)$ and
secondarily on $|V(G)|$.  The base case, $\tau(G)=0$,
is trivial, since $G$ is an independent set and we can greedily recolour $G$
from $\a$ to $\b$.  For the induction step, assume $\tau(G)\ge 1$.

We first show that $G$ contains either (i) a vertex $v$ such that $\deg(v)\le 3$
and $v$ lies in some minimum vertex cover or (ii) a vertex $v$ such that
$\deg(v)\ge 4$ and $\deg(v)-\deg^1(v)\le 2$.
Next we show how to proceed by induction in each of cases (i) and (ii).

If $G$ satisfies (a), then clearly $G$ contains an instance of (i).  Suppose $G$
satisfies (b), and consider an endblock $B$ of $G$.  If $B$ is a cycle, then $B$
contains adjacent non-cut vertices, $v$ and $w$; note that $\deg(v)=\deg(w)=2$.
Further, every vertex cover of $G$ contains $v$ or $w$, so $G$ contains (i).
Assume instead that every endblock of $G$ is an edge.  

Form a graph $J$ with a vertex for every block in $G$, where two vertices of
$J$ are adjacent if their corresponding blocks share a vertex in $G$.
Now the endpoints of a diameter in $J$ correspond with  pendent edges in $G$.
For the leaf $v$  corresponding to such a pendent edge, call its neighbour $w$.
Clearly $w$ lies in some minimum vertex cover of $G$. If $\deg(w)\le
3$, then $G$ contains (i). Otherwise, $\deg(w)-\deg^1(w)\le 2$, since by the
choice of $w$ at most one block incident to $w$ is not an endblock; so $G$
contains (ii).  This concludes the case that $G$ satisfies (b).

Assume instead that $G$ satisfies (c).  Suppose, to reach a contradiction, that
$G$ contains neither (i) nor (ii).  Form $G'$ from $G$ by deleting all vertices
with degree at most 1.  We will show that $2|E(G')|/|V(G')|\ge 2.4$, a
contradiction.  Note that $G'$ has minimum degree (at least) 2.  If
an arbitrary vertex $v$ satisfies
$\deg_{G'}(v)=2$, then $\deg_G(v)=2$, since otherwise $v$ is an instance of (i)
or (ii) in $G$.  Further, if $\deg_{G'}(v)=2$, then $\deg(w)\ge 3$ for all $vw\in
E(G')$, since otherwise $v$ or $w$ is an instance of (i) in $G$.  Now we will
reach a contradiction with discharging.  Give each vertex $v$ in $G'$ charge
$\ch(v):=d_{G'}(v)$.  We use a single discharging rule: Each 2-vertex in $G'$
takes $0.2$ from each neighbour.  Now each 2-vertex $v$ in $G'$ finishes with charge
$\ch^*(v)=2+2(0.2)=2.4$.  And each vertex $v$ with $d_{G'}(v)\ge 3$ finishes with
charge $\ch^*(v)\ge \deg_{G'}(v)-0.2\deg_{G'}(v)=0.8\deg_{G'}(v)\ge 2.4$.  This
yields the desired contradiction, which proves that $G$ contains either (i) or
(ii) if $\mad(G)<2.4$.

Now we prove the induction step in the cases that $G$ contains (i) or (ii).

Suppose $G$ contains (i).  Suppose at most one neighbour, say
$w$, of $v$ has $\a(w)$ matched with $\b(v)$. Now recolour $w$ with a colour
not matched to $\a(x)$, for every $x\in N(w)$, and not matched to $\b(v)$.
Next, recolour $v$ with $\b(v)$, and proceed on $G-v$ by induction (with the colour
matched to $\b(v)$ deleted from the lists of all vertices in $N(v)$).  
So instead assume there exist at least two such neighbours, say $w_1$ and $w_2$.  
By interchanging the roles of
$\a$ and $\b$ (and repeating this argument), we see that also there exist two
neighbours, say $x_1$ and $x_2$, such that $\b(x_1)$ and $\b(x_2)$ are matched
to $\a(v)$.  The total number of colours in $L(v)$ matched to $\a(w)$ or
$\b(w)$ for some $w\in N(v)$ is at most $2\deg(v)-2<\deg(v)+2$, since
$\deg(v)\le 3$.
Hence, there exists $c\in L(v)$ that is not matched to $\a(w)$ or $\b(w)$ for
all $w\in N(v)$.  Recolour $v$ with $c$.  Proceed on $G-v$
by induction, with that colour that $c$ is matched to deleted from the list
$L(w)$ for each $w\in N(v)$.  After finishing on $G-v$, recolour $v$ with $\b(v)$.
The number of steps we use is at most $1+n(G-v)+\tau(G-v)+1=n(G)+\tau(G)$.

Suppose instead $G$ contains (ii).  The proof is almost the same as for
(i).  If there exists $w\in N(v)$ such that $\deg(w)=1$ and $\a(v)$ is not
matched to $\b(w)$, then simply recolour $w$ with $\b(w)$, and proceed by
induction; here we use the secondary induction hypothesis, since possibly
$\tau(G-w)=\tau(G)$.  So assume that no such $w$ exists.  By interchanging the roles of
$\a$ and $\b$, we also assume that $\b(v)$ is matched to $\a(w)$ for each $w\in
N(v)$ with $d(w)=1$.  Further, if $v$ has a non-leaf neighbour $w$, then at least
one such neighbour has $\a(v)$ matched to $\b(w)$ and at least one (possibly the
same one) has $\a(w)$ matched to $\b(v)$.  
(This follows from the correspondence colouring analog of
Lemma~\ref{lem:StructureMinimalCounterexample}.)
But now there exists $(v,i)\in L(v)$
such that for all $w\in N(v)$, colour $(v,i)$ is not matched to either $\a(w)$
or $\b(w)$. Recolour $v$ to $(v,i)$, and let $G':=G-v$.  Form $(L',H)$ from
$(L,H)$ by deleting from $L(w)$, for every $w\in N(v)$, the colour matched with
$(v,i)$.  By induction, we can recolour $G'$ from $\a$ to $\b$ (both restricted
to $G'$) using at most $n(G')+\tau(G')=n(G)+\tau(G)-2$ steps.  Finally, recolour
$v$ to $\b(v)$.  This uses at most $n(G)+\tau(G)$ steps, which finishes the
proof.
\end{proof}

 
\begin{cor}
The 
\hyperref[conj:main_list]{List Conjecture} 
holds for all bipartite graphs $G$ with
$\Delta(G)\le 3$.
\end{cor}
\begin{proof}
When $G$ is bipartite, recall that $\tau(G)=\mu(G)$. 
\end{proof}
 
\section{Concluding Remarks}

In this paper, we give evidence for both the 
\hyperref[conj:main_list]{List Conjecture} 
and the
\hyperref[conj:main_DP]{Correspondence Conjecture}.
We also give evidence for an affirmative answer to Question~\ref{q:n+mu_lowerbound}. 
The \hyperref[conj:main_list]{List Conjecture} 
and Question~\ref{q:n+mu_lowerbound} would together
determine the precise diameter for $\C_{\Delta(G)+2}(G)$ and, as such, a precise
bound when Cereceda's Conjecture is restricted to regular graphs.
So we explicitly conjecture the following.

\begin{conj}[Regular Cereceda's Conjecture]\label{reg-cereceda}
For a $d$-regular graph $G$, if $k = d+2$, then 
$\diam~\C_k(G)=n(G)+\mu(G)$.
\end{conj}  

%
In Theorem~\ref{thm:lists:d+2}
we prove, 
when $\lvert L(v) \rvert \ge d(v)+2$ for all $v\in V(G)$, that
$\diam \C_L(G)\le n(G)+2\mu(G)\le 2n(G)$.
In a similar vein, it would be interesting to show that
$\diam \C_L(G)$ is linear when $\lvert L(v) \rvert \ge \lceil \mad(G)+2\rceil.$
The following conjecture can be viewed as a ``balanced'' version of the
\hyperref[conj:main_list]{List Conjecture}.

\begin{conj}[Mad Colouring Reconfiguration Conjecture]
\label{conj:main_mad}
For a graph $G$ with $\mad(G)=d$, if $L$ is a list-assignment such that $\lvert L(v) \rvert
\ge \lceil d+2\rceil$ for every $v \in V(G)$, then $\diam \C_L(G)=O_d(n).$
\end{conj}
Feghali~\cite{Feghali21} proved that if $\epsilon>0$ and $k\ge d+1+\epsilon$, then 
$\diam \C_k(G) = O_d(n (2\ln{n})^{d})$.  
Conjecture~\ref{conj:main_mad} aims to prove a similar result for list colouring,
with one more colour available for each vertex, and with a somewhat stronger
bound on diameter.
All planar graphs $G$ have $\mad(G)<6$, and
all triangle-free planar graphs $G$ have $\mad(G)<4$, so we note that
Conjecture~\ref{conj:main_mad} would imply stronger forms of~\cite[Conjecture~22]{DF21}.
If $G$ is regular, then $\mad(G)=\degen(G)$,
so Conjecture~\ref{conj:main_mad} is true by
Theorem~\ref{thm:lists:d+2}. 
Conjecture~\ref{conj:main_mad} is also related to\footnote{If $\mad(G)$ is not an
integer, then $\degen(G)+3\le \ceil{\mad(G)+2}$.  But if $G$ has a regular
subgraph of degree $\mad(G)$, then this inequality fails.} a conjecture of
Bartier et al.~\cite[Conjecture~1.6]{BBFHMP21} that graphs
$G$ with $\degen(G)=d$ satisfy $\diam \C_{d+3}(G) = O_d(n)$. 
%

We do not know if Conjecture~\ref{conj:main_mad} might be true with a linear
bound of the form $O(n)$ instead of $O_d(n).$
But we do note, for this version of the conjecture and every constant $c$,
that no bound of the form $n(G)+c \mu(G)$ can hold.
This is shown by the star $K_{1,3c+3}=(\{u\} \cup \{w_1,\ldots,w_{3c+3}\},E)$
and the colourings $\a,
\b$ with $\a(u)=1, \b(u)=4$, $\b(w_i)=1$ and $\a(w_i)= 1 + \ceil{i/(c+1)}$.
Here $\lceil \mad(G)+2\rceil=4$ and $\dist(\a,\b)>n+c.$

For a correspondence cover $(L,H)$ such that $|L(v)|=d(v)+2$ for all $v \in V(G)$, in
Theorem~\ref{factor2:thm}(ii) we proved that $\diam~\C_{(L,H)}(G)\le
n(G)+2\tau(G)$.  We view this as modest evidence
that Cereceda's Conjecture might remain true in the more
general context of correspondence colourings.

\subsection{Open Problems}

The three focuses of this paper are the 
\hyperref[conj:main_list]{List Conjecture}, 
the \hyperref[conj:main_DP]{Correspondence Conjecture},
and (to a lesser extent) Question~\ref{q:n+mu_lowerbound}. 
All of these remain open.  However, each of them seems rather hard.
So, to motivate further research, below we identify some specific classes of
graphs for which we believe that each conjecture may be approachable.
We begin with some graph classes for which it would be particularly interesting to
make further progress on the \hyperref[conj:main_list]{List Conjecture}.
\begin{enumerate}
\item Complete $r$-partite graphs for each $r\geq 3$. (We know the
\hyperref[conj:main_list]{List Conjecture}
 is true for both complete graphs and complete bipartite graphs, so this is a
natural common generalization.)
\item Bipartite graphs, not necessarily complete.
\item Outerplanar graphs and, more generally, planar graphs. 
\item Subcubic graphs.  (We already proved
the \hyperref[conj:main_DP]{Correspondence Conjecture} for this class;
nonetheless,
the \hyperref[conj:main_list]{List Conjecture} remains open.)
\end{enumerate}

Conversely,
the \hyperref[conj:main_DP]{Correspondence Conjecture} 
remains open for the following basic graph classes.

\begin{enumerate}
    \item Complete graphs. (The argument of Bonamy and Bousquet for complete
graphs directly yields the \hyperref[conj:main_list]{List Conjecture}, but does
not yield the \hyperref[conj:main_DP]{Correspondence Conjecture}.)
    \item Complete bipartite graphs. 
    \item Bipartite graphs, not necessarily complete.  (This would imply the
\hyperref[conj:main_list]{List Conjecture} for the same class.)
    \item Outerplanar graphs and, more generally, planar graphs.
\end{enumerate}

Finally, we stress that it will be interesting to improve on
Theorems~\ref{factor2:thm} and~\ref{big-lists:thm}. For example, find the smallest
$\epsilon$ such that for every graph $G$ and list-assignment $L$ with
$|L(v)| \geq d(v)+2$ for all $v\in V(G)$, we have $\diam \mathcal{C}_L(G) \leq
n(G) + (1+\epsilon)\mu(G)$. We proved $\epsilon=1$ suffices,
while $\epsilon=0$ would resolve the \hyperref[conj:main_list]{List Conjecture}.

\section*{Acknowledgement}
We thank the organisers of the online workshop \emph{Graph Reconfiguration} of the Sparse Graphs Coalition\footnote{For more information,  visit \url{https://sparse-graphs.mimuw.edu.pl/doku.php}.}, where this project started. 
We also thank the two referees for their careful reading and valuable feedback.

\bibliographystyle{habbrv}
\bibliography{reconfiguration}

\end{document}